\documentclass[a4paper,12pt,reqno,oneside]{amsart}
\usepackage[utf8x]{inputenc}
\usepackage[T1]{fontenc}
\usepackage[stretch=10]{microtype}
\usepackage{geometry}
\usepackage{lmodern}

\usepackage{amsmath}
\usepackage{amssymb}
\usepackage{amsthm}
\usepackage[shortlabels]{enumitem}
\PassOptionsToPackage{hyphens}{url}
\usepackage[bookmarksnumbered,colorlinks,linkcolor=black,citecolor=black,urlcolor=black]{hyperref}
\usepackage[capitalise]{cleveref}

\newcommand{\bA}{\mathbb{A}}
\newcommand{\bC}{\mathbb{C}}
\newcommand{\bG}{\mathbb{G}}
\newcommand{\bN}{\mathbb{N}}
\newcommand{\bP}{\mathbb{P}}
\newcommand{\bQ}{\mathbb{Q}}
\newcommand{\bR}{\mathbb{R}}
\newcommand{\bS}{\mathbb{S}}
\newcommand{\bZ}{\mathbb{Z}}

\newcommand{\cA}{\mathcal{A}}
\newcommand{\cF}{\mathcal{F}}
\newcommand{\cL}{\mathcal{L}}

\newcommand{\Ag}{\mathcal{A}_g}
\newcommand{\Hg}{\mathcal{H}_g}

\newcommand{\fS}{\mathfrak{S}}

\newcommand{\gG}{\mathbf{G}}
\newcommand{\gGL}{\mathbf{GL}}
\newcommand{\gGSp}{\mathbf{GSp}}
\newcommand{\gH}{\mathbf{H}}
\newcommand{\gSp}{\mathbf{Sp}}

\newcommand{\rH}{\mathrm{H}}
\newcommand{\rM}{\mathrm{M}}
\newcommand{\rO}{\mathrm{O}}
\newcommand{\ro}{\mathrm{o}}

\newcommand{\Ranexp}{\bR_{\mathrm{an,exp}}}

\newcommand{\Lalg}{\overline{L}}
\newcommand{\Qalg}{\overline{\bQ}}
\newcommand{\Mtmalg}{\overline{M(t_m)}}

\DeclareMathOperator{\Aut}{Aut}
\DeclareMathOperator{\Cay}{Cay}
\DeclareMathOperator{\denom}{denom}
\DeclareMathOperator{\diag}{diag}
\DeclareMathOperator{\End}{End}
\DeclareMathOperator{\Gal}{Gal}
\DeclareMathOperator{\gon}{gon}
\DeclareMathOperator{\Hom}{Hom}
\DeclareMathOperator{\Res}{Res}
\DeclareMathOperator{\Sh}{Sh}
\DeclareMathOperator{\Stab}{Stab}
\DeclareMathOperator{\trdeg}{trdeg}

\newcommand{\ad}{\mathrm{ad}}
\newcommand{\der}{\mathrm{der}}
\newcommand{\sm}{\mathrm{sm}}
\newcommand{\Zar}{\mathrm{Zar}}

\newcommand{\abs}[1]{\left\lvert #1 \right\rvert}
\newcommand{\bs}{\backslash}

\newcommand{\defterm}[1]{\textbf{#1}}

\newtheorem{lemma}{Lemma}[section]
\newtheorem{proposition}[lemma]{Proposition}
\newtheorem{theorem}[lemma]{Theorem}
\newtheorem{corollary}[lemma]{Corollary}
\newtheorem{setup}[lemma]{Set-up}
\newtheorem{conjecture}[lemma]{Conjecture}
\Crefname{conjecture}{Conjecture}{Conjectures} 

\usepackage{ifthen}
\newcounter{constant}
\newcommand{\newC}[1]{%
   \refstepcounter{constant} C_{\theconstant}%
   \ifthenelse{\equal{#1}{*}} { } {%
      \label{C:#1}%
   }%
}
\newcommand{\refC}[1]{C_{\ref*{C:#1}}}

\newcommand{\citeDRfiniteness}{\cite[Conjecture~10.3]{daw-ren:hyp-ax-schanuel}}
\newcommand{\citeDRmainthm}{\cite[Theorem~14.2]{daw-ren:hyp-ax-schanuel}}
\newcommand{\citeDRlgo}{\cite[Conjecture~11.1]{daw-ren:hyp-ax-schanuel}}
\newcommand{\citeDRfielddef}{\cite[Conjecture~12.6]{daw-ren:hyp-ax-schanuel}}

\newcommand{\citeDRuncopied}{\cite[Conjectures 10.3 and~12.7]{daw-ren:hyp-ax-schanuel}}

\title[Unlikely intersections with Hecke translates]{Unlikely intersections with Hecke translates of a special subvariety}
\author{Martin Orr}

\begin{document}

\begin{abstract}
We prove some cases of the Zilber--Pink conjecture on unlikely intersections in Shimura varieties.
Firstly, we prove that the Zilber--Pink conjecture holds for intersections between a curve and the union of the Hecke translates of a fixed special subvariety, conditional on arithmetic conjectures.
Secondly, we prove the conjecture unconditionally for intersections between a curve and the union of Hecke correspondences on the moduli space of principally polarised abelian varieties, subject to some technical hypotheses.
This generalises results of Habegger and Pila on the Zilber--Pink conjecture for products of modular curves.

The conditional proof uses the Pila--Zannier method, relying on a point-counting theorem of Habegger and Pila and a functional transcendence result of Gao.
The unconditional results are deduced from this using a variety of arithmetic ingredients: the Masser--Wüstholz isogeny theorem, comparison between Faltings and Weil heights, a super-approximation theorem of Salehi Golsefidy, and a result on expansion and gonality due to Ellenberg, Hall and Kowalski.

\end{abstract}

\renewcommand{\thefootnote}{}
\footnote{Martin Orr, Dept.\ of Mathematics, Imperial College London, South Kensington, London, UK}
\footnote{Email: m.orr@imperial.ac.uk}

\maketitle

\section{Introduction}

The aim of this paper is to prove some cases of the Zilber--Pink conjecture on unlikely intersections \cite{pink:conj}.
Pink's version of this conjecture, a generalisation of the Andr\'e--Oort and Manin--Mumford conjectures, is as follows.

\begin{conjecture} \cite[Conjecture~1.3]{pink:conj} \label{conj:zp}
Let \( S \) be a mixed Shimura variety over~\( \bC \).
Let \( V \) be an algebraic subvariety of \( S \) which is not contained in any proper special subvariety of \( S \).
Then the intersection of \( V \) with the union of all special subvarieties of \( S \) of codimension greater than \( \dim V \) is not Zariski dense in \( V \).
\end{conjecture}

In this paper, we consider cases of \cref{conj:zp} where \( S \) is a pure Shimura variety and \( V \) is an irreducible curve.
When \( V \) is an irreducible curve, the conclusion ``not Zariski dense in \( V \)'' is equivalent to ``finite.''
Instead of the intersection of~\( V \) with the union of all special subvarieties of codimension at least~\( 2 \), we consider its intersection only with Hecke translates of a fixed special subvariety.
Roughly speaking, Hecke translates mean translates of a special subvariety by rational elements of the reductive group attached to the ambient Shimura variety; see section~\ref{sec:definitions} for the precise definition.

Thus we consider the following restricted version of \cref{conj:zp}.

\begin{conjecture} \label{conj:zp-translates}
Let \( S \) be a pure Shimura variety.
Let \( S_\gH \subset S \) be a special subvariety of codimension at least \( 2 \).
Let \( V \subset S \) be an irreducible algebraic curve which is not contained in any proper special subvariety of \( S \).

Then the intersection of \( V \) with the union of all Hecke translates of \( S_\gH \) is finite.
\end{conjecture}

We prove \cref{conj:zp-translates} conditional on two arithmetic conjectures (\cref{conj:galois-orbits,conj:field-of-definition}): a large Galois orbits conjecture and a conjecture on the fields of definition of Hecke translates.
This conditional result is similar to a special case of the main theorem of \cite{daw-ren:hyp-ax-schanuel}, but not directly implied by it as we use a different measure of complexity in our conjectures.

Let \( \Ag \) denote the moduli space of principally polarised abelian varieties of dimension~\( g \).
We prove \cref{conj:zp-translates} unconditionally for Hecke correspondences (that is, Hecke translates of the diagonal) in the Shimura variety \( \Ag \times \Ag \), subject to certain technical hypotheses.
To the best of the author's knowledge, this is the first unconditional proof of cases of the Zilber--Pink conjecture for pure Shimura varieties, beyond the cases where the Shimura variety is a product of modular curves \cite{habegger-pila:beyond-ao}, \cite{pila:modular-fermat} or the special subsets have dimension zero (the André--Pink and André--Oort conjectures) \cite{orr:andre-pink}, \cite{tsimerman:galois-bounds-ag}.

The proofs are based on the Pila--Zannier strategy using o-minimality \cite{pila-zannier:manin-mumford}.
The conditional result relies on a point-counting result for ``semi-rational points'' due to Habegger and Pila \cite{habegger-pila:atypical}, a height bound for Siegel sets proved by the author of the present paper \cite{orr:siegel-heights} and a functional transcendence result of Gao \cite{gao:ax-lindemann}.
The unconditional results are deduced by proving certain cases of the arithmetic conjectures needed for the conditional result.
When the curve \( V \) is defined over \( \Qalg \), the Galois orbits conjecture is proved using the Masser--Wüstholz isogeny theorem and Faltings heights; when \( V \) is not defined over \( \Qalg \) we use results on expansion in groups due to Salehi Golsefidy \cite{sg:super-approximation} and an application of expansion to gonality of curves by Ellenberg, Hall and Kowalski \cite{ehk:expanders}.

\subsection{Unconditional results}

We prove two cases of \cref{conj:zp-translates} unconditionally.
In both cases the Shimura variety is \( S = \Ag \times \Ag \) for \( g \geq 2 \), where \( \Ag \) is the moduli space of principally polarised abelian varieties of dimension \( g \) over \( \bC \).
In both unconditional results, we look at intersections between a curve \( V \) and Hecke correspondences, that is, Hecke translates of the diagonal in \( \Ag \times \Ag \).

The two unconditional theorems are as follows.
\Cref{intro:zp-isog-asymmetric} applies when the curve \( V \) is defined over \( \Qalg \), while \cref{intro:zp-isog-transcendental} applies when \( V \) is not defined over~\( \Qalg \).
Each theorem has its own additional restrictions on \( V \) and on the subset of the union of Hecke correspondences whose intersection with \( V \) is controlled.
The term ``asymmetric curve'' in \cref{intro:zp-isog-asymmetric} essentially means that the degrees of the two projections \( V \to \Ag \) are distinct: for a precise definition, see section~\ref{ssec:asymmetric}.

\pagebreak

\begin{theorem} \label{intro:zp-isog-asymmetric} \label{zp-isog-asymmetric}
Let \( g \geq 2 \) be a positive integer.

Let \( \Sigma \) be the set of points \( (s_1, s_2) \in \Ag \times \Ag \) such that the abelian variety \( A_{s_1} \) is isogenous to \( A_{s_2} \) and \( \End A_{s_1} \cong \bZ \).

Let \( V \subset \Ag \times \Ag \) be an asymmetric irreducible algebraic curve defined over \( \Qalg \).
Suppose that \( V \) is not contained in a proper special subvariety of \( \Ag \times \Ag \).

Then \( V \cap \Sigma \) is finite.
\end{theorem}

\begin{theorem} \label{intro:zp-isog-transcendental} \label{zp-isog-transcendental}
Let \( b, g \geq 2 \) be positive integers.

Let \( \Sigma \) be the set of points \( (s_1, s_2) \in \Ag \times \Ag \) for which there exists a polarised isogeny \( (A_{s_1}, \lambda_{s_1}) \to (A_{s_2}, \lambda_{s_2}) \) whose degree is not divisible by the \( b \)-th power of any prime number.

Let \( V \subset \Ag \times \Ag \) be an irreducible algebraic curve which is not contained in a proper special subvariety of \( \Ag \times \Ag \).
Suppose that there exists an algebraically closed field \( K \subset \bC \) such that:
\begin{enumerate}[(i)]
\item The Zariski closure of \( p_1(V) \) is defined over \( K \) (where \( p_1 \) is the first projection \( \Ag \times \Ag \to \Ag \)).
\item \( V \) is not defined over \( K \).
\end{enumerate}

Then \( V \cap \Sigma \) is finite.
\end{theorem}

Hecke correspondences in \( \Ag \times \Ag \) have a natural interpretation as the set of points \( (s_1, s_2) \) such that there exists a polarised isogeny between the associated principally polarised abelian varieties \( (A_{s_1}, \lambda_{s_1}) \) and \( (A_{s_2}, \lambda_{s_2}) \).
This interpretation via isogenies is essential to the proof of \cref{intro:zp-isog-asymmetric} and so this proof is fundamentally restricted to \( \Ag \times \Ag \) or at least Shimura varieties of Hodge type, but is only used incidentally in the proof of \cref{intro:zp-isog-transcendental}.
\Cref{intro:zp-isog-transcendental} is restricted to \( \Ag \times \Ag \) because it relies on concrete calculations in the group~\( \gGSp_{2g}(\bQ) \).

\subsection{Previous results}

Previous results with analogous hypotheses were proved for the Shimura variety \( \cA_1^3 \) (where \( \cA_1 \) is the moduli space of elliptic curves): \cite[Theorem~1]{habegger-pila:beyond-ao} is analogous to \cref{intro:zp-isog-asymmetric} and \cite[Theorem~1.4]{pila:modular-fermat} is analogous to \cref{intro:zp-isog-transcendental}.
These previous results had to work with subvarieties of \( \cA_1^3 \) rather than \( \cA_1^2 \) so that there exist positive-dimensional special subvarieties of codimension at least~\( 2 \),
as is required for the intersections with a curve to be ``unlikely'' in the sense of the Pink's conjecture.
On the other hand, when \( g \geq 2 \), Pink's conjecture applies to Hecke correspondences in \( \Ag \times \Ag \) because they have codimension at least~\( 2 \).

The only special subvarieties of \( \cA_1^3 \) are intersections of subvarieties defined by one of the following two conditions:
\begin{enumerate}
\item the projection onto one of the copies of \( \cA_1 \) is a fixed special point.
\item the projection onto two of the copies of \( \cA_1 \) is a Hecke correspondence.
\end{enumerate}
Consequently, Habegger and Pila were able to prove the full Zilber--Pink conjecture for curves in \( \cA_1^3 \) satisfying the appropriate technical hypotheses.
On the other hand, when \( g \geq 2 \), \( \Ag \times \Ag \) contains special subvarieties which cannot be described in terms of just special points and Hecke correspondences (for example, subvarieties of \( \cA_2 \times \cA_2 \) which project onto quaternionic Shimura curves) and our method does not apply to these more general special subvarieties.

\Cref{intro:zp-isog-asymmetric} also implies the following special case of the André--Pink conjecture.
A more general version of this theorem was previously proved in \cite{orr:andre-pink}.

\begin{theorem} \label{andre-pink-qalg}
Let \( g \geq 2 \) be a positive integer.

Let \( s \in \Ag(\Qalg) \) be a point which is not contained in a proper special subvariety of \( \Ag \).
Let \( \Sigma \) be the set of points \( t \in \Ag \) such that \( A_t \) is isogenous to \( A_s \).

Let \( V \subset \Ag \) be an irreducible algebraic curve defined over \( \Qalg \)
which is not contained in a proper special subvariety of \( \Ag \).

Then \( V \cap \Sigma \) is finite.
\end{theorem}

\Cref{andre-pink-qalg} can be deduced from \cref{intro:zp-isog-asymmetric} by applying it to the curve \( \{ s \} \times V \subset \Ag \times \Ag \).
However \cite[Theorem~1.2]{orr:andre-pink} is more general than \cref{andre-pink-qalg} in several ways: it allows \( s \) and \( V \) to be defined over~\( \bC \) rather than \( \Qalg \)
(it is possible to use \cref{intro:zp-isog-transcendental} to prove some but not all of the cases in which \( s \) and \( V \) are not defined over \( \Qalg \)),
it allows \( s \) to be contained in a proper special subvariety,
and it applies to all curves \( V \) which are not weakly special subvarieties of \( \Ag \) (a weaker condition than not being contained in a proper special subvariety).

\subsection{Outline of paper}

In section~\ref{sec:definitions}, we define the concepts related to Shimura varieties which we shall use, along with some miscellaneous notation.
In section~\ref{sec:translates}, we prove \cref{conj:zp-translates}, conditional on arithmetic conjectures (\cref{conj:galois-orbits,conj:field-of-definition}).
\Cref{sec:hecke-correspondences} contains some results on Hecke correspondences in \( \Ag \times \Ag \) which are used in the proofs of both \cref{intro:zp-isog-asymmetric,intro:zp-isog-transcendental}, including proving \cref{conj:field-of-definition} in this case.
Finally sections \ref{sec:asymmetric} and~\ref{sec:transcendental} prove the large Galois orbits conjectures required for \cref{intro:zp-isog-asymmetric,intro:zp-isog-transcendental} respectively.

\subsection*{Acknowledgements}

The work which led to this paper began in discussions with Andrei Yafaev, to whom I am very grateful.
I would like to thank Christopher Daw, Philipp Habegger, Jonathan Pila and Jinbo Ren for useful discussions.
I am grateful to Christopher Daw and Jinbo Ren for sharing drafts of their preprint~\cite{daw-ren:hyp-ax-schanuel}, which was in preparation at the same time as this paper.
I thank the referee for their detailed reading of the paper and helpful comments.

Work on this paper was funded by European Research Council grant 307364 and by EPSRC grant EP/M020266/1.

\section{Definitions and notation} \label{sec:definitions}

We briefly recall various definitions related to Shimura varieties, in order to establish the terminology and notation which we use in this paper.
At the end of the section, we also include some miscellaneous definitions.

\subsection{Shimura varieties}

Except in the proof of \cref{ag:hecke-components-bound}, we will only work with a single geometrically connected component of a Shimura variety, which we call a ``Shimura variety component.''
We therefore omit the complexities of Deligne's adelic definition of Shimura varieties.

A \defterm{Shimura datum} is a pair \( (\gG, X) \) where \( \gG \) is a connected reductive \( \bQ \)-algebraic group and \( X \) is a \( \gG(\bR) \)-conjugacy class in \( \Hom(\bS, \gG_\bR) \) satisfying axioms 2.1.1.1--2.1.1.3 of~\cite{deligne:shimura-varieties}.
Here \( \bS \) denotes the Deligne torus \( \Res_{\bC/\bR} \bG_m \).
These axioms imply that \( X \) is a finite disjoint union of Hermitian symmetric domains \cite[Corollaire~1.1.17]{deligne:shimura-varieties}.

We select a connected component \( X^+ \subset X \).
Let \( \gG(\bR)_+ \) denote the stabiliser of~\( X^+ \) in \( \gG(\bR) \), and let \( \gG(\bR)^+ \) denote the identity connected component of \( \gG(\bR) \).
Write \( \gG(\bQ)_+ = \gG(\bQ) \cap \gG(\bR)_+ \) and \( \gG(\bQ)^+ = \gG(\bQ) \cap \gG(\bR)^+ \).

Let \( \rho \colon \gG \to \gGL_{n,\bQ} \) be a faithful representation.
For each positive integer \( N \), let
\[ \Gamma(N) = \{ \gamma \in \gG(\bQ)_+ : \rho(\gamma) \in \gGL_n(\bZ) \text{ and } \rho(\gamma) \equiv I_n \bmod N \}. \]
A \defterm{congruence subgroup} of \( \gG(\bQ)_+ \) is a subgroup which contains some \( \Gamma(N) \) as a subgroup of finite index.
The groups \( \Gamma(N) \) depend on the choice of the representation~\( \rho \), but the notion of congruence subgroup does not.

If \( \Gamma \) is a congruence subgroup of \( \gG(\bQ)_+ \), then
by \cite{baily-borel:compactification}, the quotient space \( S = \Gamma \bs X^+ \) has a canonical structure as a quasi-projective variety over~\( \bC \).
This variety~\( S \) is a connected component of the Shimura variety \( \Sh_K(\gG, X) \).
We write \( \pi \colon X^+ \to S \) for the uniformisation map.

According to Deligne's theory of canonical models (\cite{deligne:shimura-varieties}, completed in \cite{milne:canonical-models} and~\cite{borovoi:canonical-models}),
the Shimura variety~\( \Sh_K(\gG, X) \) has a canonical model over a number field.
Hence the connected component \( S \) also has a model over a number field.
We use the phrase \defterm{Shimura variety component} to mean a variety over a number field whose extension to~\( \bC \) is of the form \( \Gamma \bs X^+ \) and whose structure over a number field comes from the theory of canonical models, as described above.

\subsection{Special and weakly special subvarieties}

A \defterm{sub-Shimura datum} of \( (\gG, X) \) is a Shimura datum \( (\gH, X_\gH) \) such that \( \gH \subset \gG \) and \( X_\gH \subset X \).
Pick a connected component \( X_\gH^+ \) of \( X_\gH \) such that \( X_\gH^+ \subset X^+ \).
As a consequence of \cite[Proposition~1.15]{deligne:travaux-de-shimura}, \( \pi(X_\gH^+) \) is an algebraic subvariety of \( S \).
We call a set \( X_\gH^+ \) of this form a \defterm{pre-special subset} of \( X^+ \) and we call its image \( \pi(X_\gH^+) \) a \defterm{special subvariety} of \( S \).

There are several equivalent definitions of weakly special subvarieties of \( S \): as totally geodesic varieties \cite{moonen:linearity-I}, in terms of Shimura morphisms \cite{pink:conj} or via the following concrete but somewhat convoluted definition \cite{uy:characterization}.
Let \( (\gH, X_\gH) \subset (\gG, X) \) be a sub-Shimura datum such that the adjoint group \( \gH^\ad \) splits as a direct product \( \gH_1 \times \gH_2 \).
Then \( X_\gH^+ \) also splits as a direct product \( X_1^+ \times X_2^+ \), such that there are Shimura data \( (\gH_1, X_1) \) and \( (\gH_2, X_2) \) where \( X_1^+ \) is a connected component of \( X_1 \) and \( X_2^+ \) is a connected component of \( X_2 \).
For each point \( x_2 \in X^+_2 \), we call the fibre \( X^+_1 \times \{ x_2 \} \subset X_\gH^+ \subset X^+ \) a \defterm{pre-weakly special subset} of \( X^+ \).
Its image \( \pi(X^+_1 \times \{ x_2 \}) \) is an algebraic subvariety of \( S \),
and we call subvarieties of this form \defterm{weakly special subvarieties} of \( S \).

Note that, in the above definition, either \( \gH_1 \) or \( \gH_2 \) may be the trivial group.
If \( \gH_1 \) is trivial, then the resulting weakly special subvariety is a point (and this shows that every single-point subset of \( S \) is a weakly special subvariety).
If \( \gH_2 \) is trivial, then the resulting weakly special subvariety is \( S \) itself.

\subsection{Hecke translates and Hecke correspondences}

Let \( S_\gH = \pi(X_\gH^+) \) be a special subvariety of \( S \).
A \defterm{Hecke translate} of \( S_\gH \) is a subvariety of \( S \) of the form \( S_{\gH,\gamma} = \pi(\gamma.X_\gH^+) \) where \( \gamma \in \gG(\bQ)_+ \).
Observe that \( S_{\gH,\gamma} \) is itself a special subvariety, associated with the sub-Shimura datum \( (\gamma \gH \gamma^{-1}, \gamma.X_\gH) \subset (\gG, X) \).
The condition \( \gamma \in \gG(\bQ) \) is essential here: if \( \gamma \in \gG(\bR) \setminus \gG(\bQ) \), then \( \gamma \gH \gamma^{-1} \) need not be defined over \( \bQ \) and \( \pi(\gamma.X_\gH^+) \) is usually not a locally closed subset of \( S \) (even in the complex topology).

Consider the special case where \( \gH \) is the diagonal subgroup of \( \gG \times \gG \).
For \( \gamma \in \gG(\bQ)_+ \), \( (1, \gamma).X_\gH^+ \) is equal to the graph of the action of~\( \gamma \) on \( X^+ \).
We write
\[ T_\gamma = \pi((1, \gamma).X_\gH^+). \]
Subvarieties of \( S \times S \) of this form are called \defterm{Hecke correspondences}.


\subsection{Other definitions}

Throughout this paper, whenever we refer to polarisations, isogenies or endomorphisms of abelian varieties, we mean polarisations, isogenies or endomorphisms defined over an algebraically closed field.
When we talk about points of a variety, unless otherwise specified, we mean \( \bC \)-points.

If \( (A, \lambda) \) and \( (B, \mu) \) are principally polarised abelian varieties, then a \defterm{polarised isogeny} \( f \colon A \to B \) is an isogeny such that \( f^* \mu = n\lambda \) for some \( n \in \bZ \).

By the word \defterm{definable}, we mean ``definable in the o-minimal structure \( \Ranexp \).''

Given positive integers \( b \) and \( n \), we say that \( n \) is \defterm{\( b \)-th-power-free} if it is not divisible by the \( b \)-th power of any prime number.

Given a rational matrix \( \gamma \in \rM_n(\bQ) \) with entries \( \gamma_{ij} = a_{ij}/b_{ij} \) (each entry written as a fraction in lowest terms), we write
\begin{align*}
    \denom \gamma 
 &= \max_{1 \leq i,j \leq n} \abs{b_{ij}},
\\  \rH(\gamma)
 &= \max_{1 \leq i,j \leq n} \max(\abs{a_{ij}}, \abs{b_{ij}}).
\end{align*}

\section{Conditional Zilber--Pink for Hecke translates} \label{sec:translates}

The aim of this section is to prove \cref{conj:zp-translates} (Pink's conjecture for intersections between a curve and Hecke translates of a fixed special subvariety), conditional on a large Galois orbits conjecture and a conjecture on the fields of definition of Hecke translates.
The main theorem resembles \citeDRmainthm, but uses a different definition of complexity adapted to our special case of Hecke translates of a fixed special subvariety.

In stating the theorem and the conjectures on which it depends we shall use the following set-up.

\begin{setup} \label{zp-translates-setup}
Let \( (\gG, X) \) be a Shimura datum and let \( S \) be an associated Shimura variety component.

Let \( (\gH, X_\gH) \subset (\gG, X) \) be a sub-Shimura datum.
For each \( \gamma \in \gG(\bQ)_+ \), let \( S_{\gH,\gamma} \) denote the special subvariety \( \pi(\gamma.X_\gH^+) \subset S \).

Let \( \rho \colon \gG \to \gGL_{n,\bQ} \) be a faithful representation.
For each \( \gamma \in \gG(\bQ)_+ \), define
\[ N(\gamma) = \max(\denom \rho(\gamma), \, \abs{\det \rho(\gamma) \cdot \denom \rho(\gamma)^n}). \]

Let \( \Omega \) be a subset of \( \gG(\bQ)_+ \) and let \( \Sigma \) be a subset of \( \bigcup_{\gamma \in \Omega} S_{\gH,\gamma} \).
For each point \( s \in \Sigma \), define
\[ N(s) = \min \{ N(\gamma) : \gamma \in \Omega \text{ such that } s \in S_{\gH,\gamma} \}. \]
\end{setup}

The representation~\( \rho \) in \cref{zp-translates-setup} is an auxiliary device needed to define the complexity \( N(\gamma) \) (which can often be interpreted as a modification of \( \abs{\det \rho(\gamma)} \), taking into account denominators).
If we replace \( \rho \) by another faithful representation of~\( \gG \), the new function \( N(\gamma) \) is polynomially bounded with respect to the old \( N(\gamma) \) and vice versa.
Hence \cref{conj:galois-orbits,conj:field-of-definition} do not depend on the choice of \( \rho \), except that the constants will change.

The sets \( \Omega \) and \( \Sigma \) are treated as input data, rather than simply fixing \( \Omega = \gG(\bQ)_+ \) and \( \Sigma = \bigcup_{\gamma \in \gG(\bQ)_+} S_{\gH,\gamma} \), in order to make it clear that if we can prove \cref{conj:galois-orbits,conj:field-of-definition} for certain subsets \( \Omega \subset \gG(\bQ)_+ \) and \( \Sigma \subset \bigcup_{\gamma \in \gG(\bQ)_+} S_{\gH,\gamma} \) then we can deduce a corresponding partial version of \cref{conj:zp-translates}.
This flexibility is used in our applications to Hecke correspondences in \( \Ag \times \Ag \): in \cref{zp-isog-asymmetric}, \( \Sigma \) is restricted to points where the associated abelian varieties have endomorphism ring~\( \bZ \), while in \cref{zp-isog-transcendental}, \( \Omega \) only contains matrices with \( b \)-th-power-free determinant.
Note that the definition of \( N(s) \) depends on the set~\( \Omega \).

\begin{conjecture} \label{conj:galois-orbits}
In the situation of \cref{zp-translates-setup}, let \( V \subset S \) be an irreducible algebraic curve which is not contained in any proper special subvariety of \( S \).

Let \( L \) be a finitely generated field of characteristic zero over which \( V \) is defined.
There exist constants \( \newC{zp-translates-cond-multiplier}, \newC{zp-translates-cond-exponent} > 0 \) such that for all points \( s \in V \cap \Sigma \),
\[ \# (\Aut(\bC/L) \cdot s)  \geq  \refC{zp-translates-cond-multiplier} \, N(s)^{\refC{zp-translates-cond-exponent}}. \]
\end{conjecture}

\begin{conjecture} \label{conj:field-of-definition}
In the situation of \cref{zp-translates-setup},
for every \( \kappa > 0 \), there exists a constant \( \newC{translate-components-multiplier} \)
such that, for every \( \gamma \in \Omega \), \( S_{\gH,\gamma} \) is defined over a number field of degree at most \( \refC{translate-components-multiplier} \, N(\gamma)^\kappa \).
\end{conjecture}

\begin{theorem} \label{zp-translates-conditional}
In the situation of \cref{zp-translates-setup}, let \( V \subset S \) be an irreducible algebraic curve which is not contained in any proper special subvariety of \( S \).

Suppose that \( \dim X_\gH \leq \dim X - 2 \).
Suppose also that \cref{conj:galois-orbits,conj:field-of-definition} hold for the chosen \( (\gG, X) \), \( (\gH, X_\gH) \), \( S \), \( \rho \), \( \Omega \), \( \Sigma \) and \( V \).

Then \( V \cap \Sigma \) is finite.
\end{theorem}

The hypothesis \( \dim X_\gH \leq \dim X - 2 \) in \cref{zp-translates-conditional} ensures that intersections between \( V \) and \( S_{\gH,\gamma} \) are unlikely in the sense of Pink's conjecture (\cref{conj:zp}).

\Cref{conj:galois-orbits} is a large Galois orbits conjecture similar to others commonly used in unlikely intersections arguments.
It is analogous to \citeDRlgo{} but using a different complexity function.
Similarly, \cref{conj:field-of-definition} is analogous to \citeDRfielddef.
\Cref{zp-translates-conditional} has no conditions analogous to \citeDRuncopied{}, because the height bound we require has been proved in~\cite{orr:siegel-heights} and this is sufficient to deduce the analogue of \citeDRfiniteness{} for our setting.

The proof of \cref{zp-translates-conditional} is based on \cite[section~5]{habegger-pila:beyond-ao}, which proves a similar result in which \( S \) is a power of a modular curve.
The proof has two parts, both relying on o-minimality.
The first part is a functional transcendence argument (written geometrically), the second part is point counting using a strong variant of the Pila--Wilkie theorem.

\subsection{Intersections with a family of translates}

Let \( S \) be a Shimura variety and let \( S_\gH \subset S \) be a special subvariety of codimension at least~\( 2 \).
We show that if an algebraic curve \( V \subset S \) has positive-dimensional intersection with a semialgebraic family of translates of \( S_\gH \), then \( V \) is contained in a proper weakly special subvariety of \( S \).
If in addition \( V \) intersects a Hecke translate of \( S_\gH \), then \( V \) is contained in a proper special subvariety (the effect of this additional hypothesis resembles the fact that if a weakly special subvariety contains a special point, then it is special).

\begin{proposition} \label{translate-family-main}
Let \( S \) be a Shimura variety component associated with the Shimura datum \( (\gG, X) \).
Let \( \pi \colon X^+ \to S \) be the uniformisation map.

Let \( (\gH, X_\gH) \subset (\gG, X) \) be a sub-Shimura datum such that \( \dim X_\gH \leq \dim X - 2 \).
Let \( X_\gH^+ \) be a connected component of \( X_\gH \) which is contained in \( X^+ \).

Let \( V \subset S \) be an irreducible algebraic curve.
Consider the following hypotheses:
\begin{enumerate}[(a)]
\item There exists a connected semialgebraic set \( A \subset \gG(\bR)_+ \) of dimension at most~\( 1 \) such that
\( \pi^{-1}(V) \cap A.X_\gH^+ \)
is uncountable.
\item There exists \( \gamma \in \gG(\bQ)_+ \) such that \( \pi^{-1}(V) \cap \gamma.X_\gH^+ \) is non-empty.
\end{enumerate}

We can draw the following conclusions:
\begin{enumerate}[(i)]
\item If (a) holds, then \( V \) is contained in a proper weakly special subvariety of \( S \).
\item If (a) and (b) hold, then \( V \) is contained in a proper special subvariety of \( S \).
\end{enumerate}
\end{proposition}

Part~(i) of \cref{translate-family-main} follows from the hyperbolic Ax--Schanuel conjecture, proved by Mok, Pila and Tsimerman \cite{mpt:hyp-ax-schanuel}, but we give here a proof using a simpler functional transcendence result from \cite{gao:ax-lindemann}.
Our main interest (in order to prove to prove \cref{zp-translates-conditional}) is in part~(ii).
We do not use part~(i) in the proof of part~(ii), but we thought it is still useful to include part~(i) in order to make the roles of the two conditions (a) and~(b) clearer.

We prove \cref{translate-family-main} by constructing a complex algebraic subset of the compact dual of \( X^+ \) which contains an irreducible component of \( \pi^{-1}(V) \).
In order to obtain a weakly special subvariety from this algebraic set, we use \cite[Theorem~8.1]{gao:ax-lindemann} which is based on monodromy arguments.
Parts (i) and (ii) of \cref{translate-family-main} are then proved by controlling the dimension of this algebraic set.

Let \( \check{X} \) denote the compact dual of the Hermitian symmetric domain~\( X^+ \).
There is an action of \( \gG(\bC) \) on \( \check{X} \) given by a morphism of complex algebraic varieties
\[ \phi \colon \gG(\bC) \times \check{X} \to \check{X}. \]

Let \( A \) be the semialgebraic set which appears in hypothesis~(a) of \cref{translate-family-main}.
Let \( B \) be the Zariski closure of \( A \) inside \( \gG(\bC) \).
Because \( A \) is contained in a real algebraic set of (real) dimension at most \( 1 \), the (complex) dimension of \( B \) is at most \( 1 \).

Since \( A \subset B \), hypothesis~(a) implies that \( \pi^{-1}(V) \cap \phi(B \times X_\gH^+) \) is uncountable.
The complex analytic set \( \pi^{-1}(V) \) has only countably many irreducible components, so we can choose an irreducible component \( V_1 \subset \pi^{-1}(V) \) such that \( V_1 \cap \phi(B \times X_\gH^+) \) is uncountable.

Let \( W \) denote the Zariski closure of \( V_1 \) inside \( \check{X} \).
By \cite[Theorem~8.1]{gao:ax-lindemann}, \( W \cap X^+ \) is a pre-weakly special subset of \( X^+ \).

Let \( Y \) denote the complex algebraic subset of \( \gG(\bC) \times \check{X} \) given by
\[ Y = \{ (\beta, x) \in B \times \check{X}_\gH : \phi(\beta, x) \in W \}. \]

\begin{lemma} \label{W-phi-Y}
The Zariski closure of \( \phi(Y) \) in \( \check{X} \) is equal to \( W \).
\end{lemma}

\begin{proof}
From the definition of \( Y \), we see that \( \phi(Y) \subset W \).
Since \( W \) is Zariski closed, the Zariski closure of \( \phi(Y) \) is also contained in~\( W \).

We have
\[ V_1 \cap \phi(B \times X_\gH^+)  \subset  W \cap \phi(B \times \check{X}_\gH)  =  \phi(Y) \]
so our choice of \( V_1 \) implies that \( V_1 \cap \phi(Y) \) is uncountable.
Since \( V_1 \) is an irreducible complex analytic curve, it follows that \( V_1 \) is contained in the Zariski closure of \( \phi(Y) \).
Therefore \( W = V_1^\Zar \) is contained in the Zariski closure of \( \phi(Y) \).
\end{proof}

The proof of \cref{translate-family-main}(i) is now immediate.

\begin{proof}[Proof of \cref{translate-family-main}(i)]
\( V \) is contained in \( \pi(W) \), which is a weakly special subvariety of \( S \).
So it suffices to show that \( \dim W < \dim S = \dim X \).

Since \( \phi \) is a morphism of algebraic varieties and using \cref{W-phi-Y}, we have \( \dim W \leq \dim Y \).
Hence using the hypothesis that \( \dim X_\gH \leq \dim X - 2 \), we get
\[ \dim W  \leq  \dim Y  \leq  \dim B + \dim \check{X}_\gH  \leq  1 + (\dim X - 2)  =  \dim X - 1.
\qedhere
\]
\end{proof}

We will use \cref{W-phi-Y} to prove part~(ii) of \cref{translate-family-main} (we do not use part~(i) in the proof of part~(ii)).
First we use hypothesis~(b) to obtain a more refined bound for \( \dim Y \).

Assume for contradiction that \( V \) is not contained in any proper special subvariety of \( S \).
It follows that the pre-weakly special set \( W \cap X^+ \) is not contained in any proper pre-special subset of \( X^+ \).
Consequently, using the notation from the definition of pre-weakly special subsets (applied to \( W \cap X^+ \)), we must have \( (\gH, X_\gH) = (\gG, X) \).
Hence there is a direct product decomposition \( X^+ = X_1^+ \times X_2^+ \) such that
\[ W \cap X^+ = X_1^+ \times \{ x_2 \} \]
for some point \( x_2 \in X_2^+ \).
Write \( \gG^\ad = \gG_1 \times \gG_2 \) for the associated decomposition of the adjoint group.


\begin{lemma} \label{dim-Y-upper-bound}
In the setting of \cref{translate-family-main}, assume that (a) and (b) hold, and that \( V \) is not contained in a proper special subvariety of~\( S \).

Then \( \dim Y < \dim X_1^+ \).
\end{lemma}

\begin{proof}
Let \( p_2 \) denote the projection \( X^+ \to X_2^+ \).

We have \( p_2(V_1) = \{ x_2 \} \) and thus \( p_2(\pi^{-1}(V)) \subset p_2(\Gamma.x_2) \).
Therefore hypothesis~(b) from \cref{translate-family-main} implies that \( x_2 \in p_2(\gamma.X_\gH^+) \) for some \( \gamma \in \gG(\bQ)_+ \).
Because \( \gamma \in \gG(\bQ)_+ \), \( p_2(\gamma.X_\gH^+) \) is a pre-special subvariety of \( X_2^+ \).
Since \( V \) is not contained in any proper special subvariety of \( S \), we deduce that \( p_2(\gamma.X_\gH^+) = X_2^+ \).
It follows that \( \gamma \gH \gamma^{-1} \) projects surjectively onto \( \gG_2 \), and hence also \( \gH \) projects surjectively onto \( \gG_2 \).

For each \( \beta \in B \), consider
\begin{align*}
    Y_\beta
  & =  \{ x \in \check{X}_\gH : (\beta, x) \in Y \}
\\& =  \{ x \in \check{X}_\gH : \phi(\beta, x) \in \check{X}_1 \times \{ x_2 \} \}
\\& =  \{ x \in \check{X}_\gH : p_2(\beta.x) = x_2 \}
\\& =  \{ x \in \check{X}_\gH : p_2(x) = p_2(\beta)^{-1}.x_2 \}.
\end{align*}
In other words, \( Y_\beta \) is the fibre of \( p_{2|\check{X}_\gH} \colon \check{X}_\gH \to \check{X}_2 \) above the point \( p_2(\beta)^{-1}.x_2 \in \check{X}_2 \).

We have seen that \( \gH \) surjects onto \( \gG_2 \).
Hence \( p_{2|\check{X}_\gH} \colon \check{X}_\gH \to \check{X}_2 \) is surjective,
so a general fibre of \( p_{2|\check{X}_\gH} \) has dimension equal to \( \dim \check{X}_\gH - \dim \check{X}_2 \).
Since \( p_{2|\check{X}_\gH} \) is an equivariant morphism of \( \gH(\bC) \)-homogeneous spaces, all its fibres have the same dimension.
Using the hypothesis on \( \dim X_\gH \) from \cref{translate-family-main}, we get
\[ \dim Y_\beta  \leq  \dim \check{X}_\gH - \dim \check{X}_2^+  \leq  (\dim X^+ - 2) - \dim X_2^+  =  \dim X_1^+ - 2. \]
This inequality holds for all \( \beta \in B \), while a general fibre of \( Y \to B \) has dimension at least \( \dim Y - \dim B \).
We conclude that
\[ \dim Y  \leq  (\dim X_1^+ - 2) + \dim B  \leq  \dim X_1^+ - 1.
\qedhere
\]
\end{proof}

\begin{proof}[Proof of \cref{translate-family-main}(ii)]
Since \( W = \check{X}_1 \times \{ x_2 \} \), we have \( \dim W = \dim X_1^+ \).
By \cref{W-phi-Y}, \( \dim W  \leq  \dim Y \).
This contradicts \cref{dim-Y-upper-bound} unless \( V \) is contained in a proper special subvariety of \( S \).
\end{proof}

\subsection{Point counting}

In order to prove \cref{zp-translates-conditional}, we apply a variant of the Pila--Wilkie counting theorem \cite[Corollary~7.2]{habegger-pila:atypical} to the set
\[ Z = \{ (\gamma, x) \in \gG(\bR)_+ \times \cF : x \in \gamma.X_\gH^+ \text{ and } \pi(x) \in V \} \]
where \( \cF \) is a suitable fundamental set in \( X^+ \).
\Cref{conj:galois-orbits,conj:field-of-definition}, together with a height bound from \cite{orr:siegel-heights}, imply that \( Z \) contains many points \( (\gamma, x) \) for which \( \gamma \) is rational and has bounded height.
By \cite[Corollary~7.2]{habegger-pila:atypical}, this implies that \( Z \) contains a definable path whose image in \( \gG(\bR)_+ \) is semialgebraic and whose image in \( X^+ \) is positive-dimensional.
We conclude by applying \cref{translate-family-main}.

\begin{lemma} \label{fundamental-sets}
There exist fundamental sets \( \cF \subset X^+ \) (for the action of~\( \Gamma \)) and \( \cF_\gH \subset X_\gH^+ \) (for the action of \( \Gamma_\gH = \Gamma \cap \gH(\bQ)_+ \)) such that \( \cF_\gH \subset \cF \).
Furthermore \( \pi_{|\cF} \colon \cF \to S \) is definable in \( \Ranexp \).
\end{lemma}

\begin{proof}
We define a \defterm{Siegel set} in \( \gG(\bR)^+ \) to be the intersection of \( \gG(\bR)^+ \) with a Siegel set in \( \gG(\bR) \), as defined in \cite[section~2]{orr:siegel-heights}.
We define Siegel sets in \( \gH(\bR)^+ \) analogously.

Choose a point \( x_0 \in X_\gH^+ \).
The stabiliser of \( x_0 \) in \( \gH(\bR)^+ \) is a maximal compact subgroup \( K_\gH \subset \gH(\bR)^+ \), so we can identify \( X_\gH^+ \) with \( \gH(\bR)^+ / K_\gH \).
Choose a \( K_\gH \)-right-invariant Siegel set \( \fS_\gH \subset \gH(\bR)^+ \).
By \cite[Théorèmes 13.1 and~15.4]{borel:groupes-arithmetiques}, there exists a finite set \( C_\gH \subset \gH(\bQ)^+ \) such that
\[ \cF_\gH = C_\gH.\fS_\gH.x_0 \]
is a fundamental set in \( X_\gH^+ \) for the action of \( \Gamma_\gH \).

By \cite[Theorem~1.2]{orr:siegel-heights}, we can find a Siegel set \( \fS_\gG \subset \gG(\bR)^+ \) and a finite set \( C \subset \gG(\bQ)^+ \) such that
\[ \fS_\gH \subset C.\fS_\gG. \]
According to \cite[Theorem~4.1]{orr:siegel-heights}, \( \fS_\gG \) is \( K_\gG \)-right-invariant where \( K_\gG \) is some maximal compact subgroup of \( \gG(\bR)^+ \) which contains \( K_\gH \).
Because \( X_\gH^+ \to X^+ \) is injective,
the stabiliser in \( \gG(\bR)^+ \) of \( x_0 \in X^+ \) is the unique such subgroup \( K_\gG \).
Hence by \cite[Théorèmes 13.1 and~15.4]{borel:groupes-arithmetiques}, there exists a finite set \( C_\gG \subset \gG(\bQ)^+ \) such that
\( C_\gG.\fS_\gG.x_0 \)
is a fundamental set in \( X^+ \) for the action of \( \Gamma \).

Let \( C_\gG' = C_\gG \cup C_\gH.C \), which is a finite subset of \( \gG(\bQ)^+ \).
Then
\[ \cF = C_\gG'.\fS_\gG.x_0 \]
is a fundamental set in \( X^+ \) for~\( \Gamma \) and satisfies \( \cF_\gH \subset \cF \).

The restriction of~\( \pi \) to \( \cF \) is definable in \( \Ranexp \) by \cite[Theorem~4.1]{kuy:ax-lindemann}.
\end{proof}

\pagebreak

In order to relate the complexity \( N(s) \) of a point \( s \in V \cap S_{\gH,\gamma} \) to the height of~\( \gamma \),
we use \cite[Theorem~1.1]{orr:siegel-heights} as follows.

\begin{lemma} \label{siegel-height-bound}
In the situation of \cref{zp-translates-setup},
there exists a constant \( \newC{siegel-height-multiplier} \) such that, for every \( s \in \Sigma \), there exist \( \gamma \in \gG(\bQ)_+ \) and \( x \in \cF \cap \gamma.\cF_\gH \) satisfying
\[ \pi(x) = s  \quad \text{and} \quad  \rH(\rho(\gamma))  \leq  \refC{siegel-height-multiplier} \, N(s). \]
\end{lemma}

\begin{proof}
By the definition of \( N(s) \), there exists \( \gamma_1 \in \Omega \subset \gG(\bQ)_+ \) such that \( s \in S_{\gH,\gamma_1} \) and \( N(s) = N(\gamma_1) \).
Because \( \gamma_1.\cF_\gH \) is a fundamental set for the action of \( \gamma_1 \Gamma_\gH \gamma_1^{-1} \) on \( \gamma_1.X_\gH^+ \),
we can pick \( x_1 \in \gamma_1.\cF_\gH \) such that \( s = \pi(x_1) \).

Because \( \cF \) is a fundamental set for the action of \( \Gamma \) on \( X^+ \), there exists \( \gamma_2 \in \Gamma \) such that \( \gamma_2.x_1 \in \cF \).
(Note that we cannot just pick \( \gamma_2 = \gamma_1^{-1} \) because \( \gamma_1 \) is not necessarily in \( \Gamma \).)

Since \( \gamma_2.x_1 \) and \( \gamma_1^{-1}.x_1 \) are both in \( \cF = C_\gG'.\fS_\gG.x_0 \), we have
\[ \gamma_2.\gamma_1  \in  C_\gG'.\fS_\gG.\fS_\gG^{-1}.C_\gG'^{-1}. \]
In other words, there exist \( \xi_1, \xi_2 \in C_\gG' \) such that
\[ \xi_1^{-1} \gamma_2 \gamma_1 \xi_2  \in  \fS_\gG.\fS_\gG^{-1}. \]

Since \( \gamma_2 \) is in \( \Gamma \), it has determinant \( \pm 1 \) and bounded denominators.
Since \( \xi_1 \), \( \xi_2 \) are elements of the fixed finite set \( C_\gG' \), their determinants and denominators are bounded.
Hence \( N(\xi_1^{-1} \gamma_2 \gamma_1 \xi_2) \) is bounded by a constant multiple of \( N(\gamma_1) = N(s) \).

By \cite[Theorem~1.1]{orr:siegel-heights}, \( \rH(\rho(\xi_1^{-1} \gamma_2 \gamma_1 \xi_2)) \) is bounded by a constant multiple of \( N(\xi_1^{-1} \gamma_2 \gamma_1 \xi_2) \).
Again since \( \xi_1 \), \( \xi_2 \) are elements of a fixed finite set, \( \rH(\rho(\gamma_2 \gamma_1)) \) is similarly bounded.

Letting \( x = \gamma_2.x_1 \) and \( \gamma = \gamma_2.\gamma_1 \) completes the proof.
\end{proof}

As a first application of \cref{siegel-height-bound}, we show that the complexity of points in \( V \cap \Sigma \) tends to infinity in the situation of \cref{zp-translates-conditional}.
This lemma uses only the existence of a bound for \( \rH(\rho(\gamma)) \) in terms of \( N(s) \) from \cref{siegel-height-bound} and does not require that the bound is polynomial.
The fact that the bound is polynomial will be used later, in the proof of \cref{zp-translates-conditional}.

\begin{lemma} \label{complexity-unbounded}
In the situation of \cref{zp-translates-setup},
let \( V \subset S \) be an irreducible algebraic curve which is not contained in any proper special subvariety of~\( S \).

For each positive integer \( N \), the set
\[ \{ s \in V \cap \Sigma : N(s) \leq N \} \]
is finite.
\end{lemma}

\begin{proof}
If \( s \in V \cap \Sigma \) and \( N(s) \leq N \), then \cref{siegel-height-bound} tells us that \( s \in V \cap S_{\gH,\gamma} \) for some \( \gamma \in \gG(\bQ)_+ \) such that
\( \rH(\rho(\gamma))  \leq  \refC{siegel-height-multiplier} N \).
There are finitely many \( \gamma \in \gG(\bQ)_+ \) satisfying this height bound.

For each \( \gamma \in \gG(\bQ)_+ \), \( S_{\gH,\gamma} \) is an algebraic variety
while \( V \) is an irreducible algebraic curve not contained in \( S_{\gH,\gamma} \).
Hence \( V \cap S_{\gH,\gamma} \) is finite for each \( \gamma \).

Therefore the set described in the lemma is a finite union of finite sets.
\end{proof}

We use the following point counting result of Habegger and Pila, which is a simplified version of \cite[Corollary~7.2]{habegger-pila:atypical}.

\begin{theorem}  \label{habegger-pila-counting}
Let \( Z \subset \bR^m \times \bR^n \) be a definable set.
Let \( p_1 \colon \bR^m \times \bR^n \to \bR^m \) and \( p_2 \colon \bR^m \times \bR^n \to \bR^n \) denote the projection maps.
For \( T \geq 1 \), let
\[ Z^{\sim}(\bQ, T) = \{ z \in Z : p_1(z) \in \bQ^m \text{ and } \rH(p_1(z)) \leq T \}. \]

For every \( \epsilon > 0 \), there exists a constant \( \newC{counting-multiplier} \) (depending on \( Z \) and \( \epsilon \)) with the following property:
if there exists \( T \geq 1 \) such that
\[ \#p_2(Z^{\sim}(\bQ, T)) \geq \refC{counting-multiplier} \, T^\epsilon \]
then there exists a continuous definable path \( [0,1] \to Z \) such that:
\begin{enumerate}[(i)]
\item the composition \( p_1 \circ \beta \colon [0,1] \to \bR^m \) is semialgebraic and its restriction to \( (0,1) \) is real analytic;
\item the composition \( p_2 \circ \beta \colon [0,1] \to \bR^n \) is non-constant.
\end{enumerate}
\end{theorem}

We will apply \cref{habegger-pila-counting} to the set
\[ Z = \{ (\gamma, x) \in \gG(\bR)_+ \times \cF : x \in \gamma.X_\gH^+ \text{ and } \pi(x) \in V \}. \]
This set is definable by \cref{fundamental-sets}.

\begin{proof}[Proof of \cref{zp-translates-conditional}]
Assume for contradiction that \( V \cap \Sigma \) is infinite.
According to \cref{complexity-unbounded}, \( V \cap \Sigma \) contains points of arbitrarily large complexity.

Let \( s \) be a point in \( V \cap \Sigma \) of large complexity (we will decide how large later).
Choose \( \gamma \in \Omega \) such that \( s \in V \cap S_{\gH,\gamma} \) and \( N(\gamma) = N(s) \).

Let \( \refC{zp-translates-cond-multiplier} \), \( \refC{zp-translates-cond-exponent} \), \( \refC{translate-components-multiplier} \) be the constants from \cref{conj:galois-orbits,conj:field-of-definition}.
By \cref{conj:field-of-definition}, \( S_{\gH,\gamma} \) is defined over an extension \( L'/L \) of degree at most \( \refC{translate-components-multiplier} \, N(s)^{\refC{zp-translates-cond-exponent}/2} \).
By \cref{conj:galois-orbits}, we have
\[ \# (\Aut(\bC/L') \cdot s)  \geq  \frac{\# (\Aut(\bC/L') \cdot s)} {[L':L]}  \geq  \frac{\refC{zp-translates-cond-multiplier}} {\refC{translate-components-multiplier}} \, N(s)^{\refC{zp-translates-cond-exponent}/2}. \]


Since the varieties \( V \) and \( S_{\gH,\gamma} \) are both defined over \( L' \),
every point \( s' \) in the Galois orbit \( \Aut(\bC/L') \cdot s \) lies in \( V \cap S_{\gH,\gamma} \).
Therefore \( N(s') \leq N(s) \).
By \cref{siegel-height-bound}, there exist \( \gamma' \in \gG(\bQ)_+ \) and \( x' \in \cF \cap \gamma'.\cF_\gH \) such that \( \pi(x') = s' \) and
\[ \rH(\rho(\gamma'))  \leq  \refC{siegel-height-multiplier} \, N(s')  \leq  \refC{siegel-height-multiplier} \, N(s). \]
Thus \( (\gamma', x') \in Z^\sim(\bQ, \refC{siegel-height-multiplier} \, N(s)) \).
The fact that \( \pi(x') = s' \) implies that distinct points \( s' \in \Aut(\bC/L') \cdot s \) give rise to distinct points \( x' \in p_2(Z^\sim(\bQ, \refC{siegel-height-multiplier} \, N(s)) \).

We conclude that if we take \( T = \refC{siegel-height-multiplier} \, N(s) \), then
\[ \# p_2(Z^\sim(\bQ, T)
   \geq  (\refC{zp-translates-cond-multiplier} / \refC{translate-components-multiplier}) \, N(s)^{\refC{zp-translates-cond-exponent}/2}
   =  (\refC{zp-translates-cond-multiplier} / \refC{translate-components-multiplier}) \, (T / \refC{siegel-height-multiplier})^{\refC{zp-translates-cond-exponent}/2}. \]
Taking \( \epsilon = \refC{zp-translates-cond-exponent}/3 \) in \cref{habegger-pila-counting}, we see that the inequality \( \# p_2(Z^\sim(\bQ), T)  \geq  \refC{counting-multiplier} \, T^\epsilon \) will be satisfied for large \( T \) (and we can make \( T \) arbitrarily large by picking \( s \in V \cap \Sigma \) with \( N(s) \) large enough).
Therefore there exists a continuous definable path \( \beta \colon [0,1] \to Z \) with the properties listed in \cref{habegger-pila-counting}.

Let \( A \subset \gG(\bR) \) denote the image of \( p_1 \circ \beta \).
By property~(i) from \cref{habegger-pila-counting}, \( A \) is semialgebraic.
Since \( A \) is the image of a path, \( \dim A \leq 1 \).

By the definition of \( Z \), the image of \( p_2 \circ \beta \) is contained in \( A.X_\gH^+ \cap \pi^{-1}(V) \).
By property~(ii) from \cref{habegger-pila-counting}, \( p_2 \circ \beta \) is non-constant.
Since \( p_2 \circ \beta \) is a definable path, we deduce that the image of \( p_2 \circ \beta \) is uncountable.
Hence hypothesis~(a) of \cref{translate-family-main} is satisfied.

Since we are assuming that \( V \cap \Sigma \) is non-empty, hypothesis~(b) of \cref{translate-family-main} is also satisfied.
Therefore we can apply \cref{translate-family-main}(ii) to conclude that \( V \) is contained in a proper special subvariety of \( S \), which gives a contradiction.
\end{proof}

\section{Degree and field of definition of Hecke correspondences in \texorpdfstring{\( \Ag \)}{Ag}} \label{sec:hecke-correspondences}

In order to apply \cref{zp-translates-conditional} to Hecke correspondences in \( \Ag \times \Ag \),
we need several lemmas about these Hecke correspondences.
In this section, we prove \cref{conj:field-of-definition} for Hecke correspondences in \( \Ag \times \Ag \) and bound the degree of such a Hecke correspondence in terms of the complexity of an associated general symplectic matrix.

All of the proofs in this section rely on explicit calculations with matrices in \( \gGSp_{2g}(\bQ) \) and in particular the symplectic elementary divisor theorem.
It seems plausible that these results can be generalised to arbitrary Shimura varieties, albeit with more difficult proofs.
If so, \cref{zp-isog-transcendental} could be generalised to all Shimura varieties as the results of this section are the only reason it is restricted to~\( \Ag \).
On the other hand, generalising the results of this section to other Shimura varieties would not allow us to immediately generalise \cref{zp-isog-asymmetric} because its proof uses abelian varieties much more fundamentally, via the Faltings height.

Let \( \nu \colon \gGSp_{2g} \to \bG_m \) denote the standard character of \( \gGSp_{2g} \).
Thus the action of \( \gGSp_{2g}(k) \) on the standard symplectic form \( \psi \colon k^{2g} \times k^{2g} \to k \) is given by
\[ \psi(\gamma.x, \gamma.y) = \nu(\gamma) \, \psi(x, y). \]
Let \( \Gamma = \gSp_{2g}(\bZ) \),
let \( \Gamma(m) = \ker(\gSp_{2g}(\bZ) \to \gSp_{2g}(\bZ/m\bZ)) \) for each positive integer~\( m \),
and let \( \Gamma_\gamma = \Gamma \cap \gamma^{-1} \Gamma \gamma \) for each \( \gamma \in \gGSp_{2g}(\bQ)_+ \).

We shall repeatedly use the following symplectic elementary divisor theorem.

\begin{lemma}  \cite[Lemma~3.3.6]{andrianov:quadratic-forms}  \label{symplectic-elem-div}
Let \( \gamma \in \gGSp_{2g}(\bQ)_+ \cap \rM_{2g}(\bZ) \).
Then \( \gamma \) can be written in the form \( \kappa \delta \lambda \) where \( \kappa, \lambda \in \Gamma \) and
\[ \delta = \diag(a_1, \dotsc, a_g, b_1, \dotsc, b_g), \]
with \( a_i, b_i \in \bZ_{>0} \), \( a_i b_i = \nu(\gamma) \) for all \( i \), \( a_i \vert a_{i+1} \) for each \( i \leq g-1 \) and \( a_g \vert b_g \).
Furthermore these conditions determine \( \delta \) uniquely.
\end{lemma}

\subsection{Hecke correspondences and isogenies}

We begin by recalling a well-known description of Hecke correspondences in \( \Ag \times \Ag \) as moduli of pairs of abelian varieties related by a polarised isogeny.

\begin{lemma} \label{ag:isogeny-to-hecke}
Let \( s = (s_1, s_2) \in \Ag \times \Ag \) and let \( n \) be a positive integer.
The following are equivalent:
\begin{enumerate}
\item There exists a polarised isogeny from \( (A_{s_1}, \lambda_{s_1}) \) to \( (A_{s_2}, \lambda_{s_2}) \) of degree \( n \).
\item \( s \in T_\gamma \) for some \( \gamma \in \gGSp_{2g}(\bQ)_+ \cap \rM_{2g}(\bZ) \) such that \( \det \gamma = n \).
\end{enumerate}
\end{lemma}

\begin{proof}
The point \( s \in \Ag \times \Ag \) lies in \( T_\gamma \) if and only if \( \gamma \) is the rational representation of a polarised isogeny \( f \colon A_{s_1} \to A_{s_2} \) with respect to some symplectic bases of \( H_1(A_{s_1}, \bZ) \) and \( H_1(A_{s_2}, \bZ) \).
When such an \( f \) exists, \( \deg f = \det \gamma \).
\end{proof}

\subsection{Field of definition of Hecke correspondences}

We prove that \cref{conj:field-of-definition} holds for Hecke correspondences in \( \Ag \times \Ag \).
Note that in the case of~\( \Ag \), every Hecke correspondence has the same field of definition so the dependence on \( N(\gamma) \) in \cref{conj:field-of-definition} is not required, but we cannot expect this to hold for Hecke correspondences on arbitrary Shimura varieties.


\begin{lemma} \label{ag:hecke-components-bound}
Every Hecke correspondence in \( \Ag \times \Ag \) is defined over \( \bQ \).
\end{lemma}

\begin{proof}
Let \( K = \gSp_{2g}(\hat\bZ) \).
For each \( \gamma \in \gGSp_{2g}(\bQ)_+ \), let \( K_\gamma = K \cap \gamma^{-1} K \gamma \).

The Hecke correspondence \( T_\gamma \subset \Ag \times \Ag \) is isomorphic to a connected component of the Shimura variety \( \Sh_\gamma = \Sh_{K_\gamma}(\gGSp_{2g}, \Hg^{\pm}) \).
By Deligne's theory of canonical models of Shimura varieties, \( \Sh_\gamma \) is defined over the reflex field of the Shimura datum \( (\gGSp_{2g}, \Hg^{\pm}) \), namely \( \bQ \).
Hence it suffices to show that \( \Sh_\gamma \) is connected.

Since \( \gGSp_{2g}^\der = \gSp_{2g} \) is simply connected, we can apply \cite[2.7]{deligne:travaux-de-shimura}.
This tells us that the connected components of \( \Sh_\gamma \) are in bijection with
\[ \bQ^\times \bs \bA_f^\times \times \bR^\times / \nu(K_\gamma \times K_\infty) \]
where \( K_\infty = \Stab_{\gGSp_{2g}(\bR)}(x) \) for some point \( x \in X \).
Now \( K_\infty \) is the product of the centre \( \bR^\times \) with a maximal compact subgroup of \( \gGSp_{2g}(\bR)_+ \), so \( \nu(K_\infty) = \bR_+^\times \).
We deduce that the connected components of \( \Sh_\gamma \) are in bijection with
\[ \bQ^\times_+ \bs \bA_f^\times / \nu(K_\gamma). \]

Write \( \gamma = \kappa \delta \lambda \) as in \cref{symplectic-elem-div}.
Then \( K_\gamma = \lambda^{-1} K_\delta \lambda \).
Since the codomain of \( \nu \) is commutative, it follows that \( \nu(K_\gamma) = \nu(K_\delta) \).

Since \( \delta \) is diagonal, \( K_\delta \) contains all the diagonal elements of \( K \).
In particular, \( K_\delta \) contains \( \alpha_x = \diag(x, \dotsc, x, 1, \dotsc, 1) \) for every \( x \in \hat\bZ^\times \).
We have \( \nu(\alpha_x) = x \) for each \( x \in \hat\bZ^\times \) and so \( \nu(K_\delta) = \hat\bZ^\times \).

Thus the connected components of \( \Sh_\gamma \) are in bijection with \( \bQ^\times_+ \bs \bA_f^\times / \hat\bZ^\times \).
This is isomorphic to the class group of \( \bQ \), so has one element.
\end{proof}

\subsection{Degree of Hecke correspondences}

In order to prove \cref{zp-isog-transcendental}, we will need a lower bound for the degree of the finite morphism \( T_\gamma \to \Ag \) (the restriction of first projection \( \Ag \times \Ag \to \Ag \)) in terms of \( \nu(\gamma) \).
The degree of this morphism can also be described as the index \( [\Gamma : \Gamma_\gamma] \).

\begin{lemma} \label{ag:index-complexity-bound}
Let \( \gamma \in \gGSp_{2g}(\bQ)_+ \cap \rM_{2g}(\bZ) \).
Suppose that the entries of \( \gamma \) have no common factor.
Then
\[ [\Gamma : \Gamma_\gamma]  \geq  \nu(\gamma). \]
\end{lemma}

\begin{proof}
Write \( \gamma = \kappa \delta \lambda \) as in \cref{symplectic-elem-div}.
Then \( \Gamma_\gamma = \lambda^{-1} \Gamma_\delta \lambda \) and \( \lambda \in \Gamma \), so
\( [\Gamma : \Gamma_\gamma] = [\Gamma : \Gamma_\delta] \).
Because \( \kappa, \lambda \in \Gamma \), the entries of \( \delta \) still have no common factor.
Therefore we can replace \( \gamma \) by \( \delta \) without loss of generality, so that \( \gamma = \diag(a_1, \dotsc, a_g, b_1, \dotsc, b_g) \) for integers which satisfy the conditions of \cref{symplectic-elem-div}.
Because the entries of \( \gamma \) have no common factor, we must have \( a_1 = 1 \).

Factorise \( \nu(\gamma) \) as \( \nu(\gamma) = q_1 q_2 \dotsm q_r \), where the \( q_i \) are powers of distinct primes.
For each \( i = 1, \dotsc, r \), let \( G_{\gamma,i} \) denote the image of \( \Gamma_\gamma \) in \( \gSp_{2g}(\bZ/q_i\bZ) \).
By \cite[Theorems 1 and~4]{newman-smart}, the map \( \Gamma \to \prod_i \gSp_{2g}(\bZ/q_i\bZ) \) is surjective.
Because the image of \( \Gamma_\gamma \) under this map is contained in \( \prod_i G_{\gamma,i} \), we deduce that
\[ [\Gamma : \Gamma_\gamma]  \geq  \prod_{i=1}^r [\gSp_{2g}(\bZ/q_i\bZ) : G_{\gamma,i}]. \]
Hence it suffices to prove that
\[ [\gSp_{2g}(\bZ/q_i\bZ) : G_{\gamma,i}]  \geq  q_i \]
for each \( i \).
We will prove this by exhibiting \( q_i \) elements \( \eta_0, \dotsc, \eta_{q_i-1} \in \gSp_{2g}(\bZ/q_i\bZ) \) which lie in distinct left \( G_{\gamma,i} \)-cosets.

For each integer \( \ell \), let \( \eta_\ell \in \gSp_{2g}(\bZ/q_i\bZ) \) be the matrix with ones on the diagonal, \( \ell \) in position \( (g+1,1) \) and zeroes everywhere else.
We can calculate \( \eta_k^{-1} \eta_\ell = \eta_{\ell-k} \).

Since \( a_1 b_1 = n \) and \( a_1 = 1 \), we deduce that \( b_1 = n \equiv 0 \bmod q_i \).
It follows that every element of \( G_{\gamma,i} \) has zero (mod \( q_i \)) as its entry in position \( (g+1,1) \).
Hence if \( \eta_k^{-1} \eta_\ell \in G_{\gamma,i} \), we must have \( k \equiv \ell \bmod q_i \).

Thus \( \eta_0, \dotsc, \eta_{q_i-1} \) lie in distinct left \( G_{\gamma,i} \)-cosets of \( \gSp_{2g}(\bZ/q_i\bZ) \).
\end{proof}

\section{Unconditional Zilber--Pink: asymmetric curves} \label{sec:asymmetric}

In this section, we prove \cref{conj:galois-orbits} (large Galois orbits) for endomorphism-generic points in the intersection of Hecke correspondences in \( \Ag \times \Ag \) with an asymmetric curve defined over \( \Qalg \) (the term ``asymmetric curve'' is defined below).
We use this to deduce \cref{intro:zp-isog-asymmetric}.

The proof of this case of \cref{conj:galois-orbits} uses the Weil height machine, Faltings heights and the Masser--Wüstholz isogeny theorem.
The use of heights means that the argument applies only to curves defined over \( \Qalg \) (some of the arguments could be extended to work over~\( \bC \) by using Moriwaki's height, but there is no version of Masser--W\"ustholz known for Moriwaki's height).
The use of Faltings heights and the Masser--Wüstholz theorem limit the method to Shimura varieties which have interpretations in terms of moduli of abelian varieties.

Let \( \Sigma \) denote the set defined in \cref{zp-isog-asymmetric}.
Because \( \Sigma \) only contains points \( (s_1, s_2) \) where \( \End A_{s_1} \cong \bZ \), it is not a union of special subvarieties of \( \Ag \times \Ag \).
Instead \( \Sigma \) can be obtained from the union of the Hecke correpondences in \( \Ag \times \Ag \) by removing countably many smaller special subvarieties (parametrising abelian varieties with endomorphism rings larger than~\( \bZ \)).
This restriction on \( \End A_{s_1} \) is necessary because our Galois orbits bound is in terms of the smallest degree of an isogeny \( A_{s_1} \to A_{s_2} \).
If \( \End A_{s_1} \not\cong \bZ \), then the isogeny of minimum degree might not be polarised, and this would prevent us applying \cref{ag:isogeny-to-hecke} to these isogenies.

The restriction to asymmetric curves in \cref{zp-isog-asymmetric} is used to bound the height of points in \( V \cap \Sigma \) in terms of their complexity (\cref{faltings-height-bound}) and thence to obtain a Galois orbits bound which is uniform for all points in \( V \cap \Sigma \).

\subsection{Asymmetric curves and Faltings heights} \label{ssec:asymmetric}

We say that a curve \( V \subset \Ag \times \Ag \) is \defterm{asymmetric} if the two projections \( V \to \Ag \) have different degrees (the definition we are using for the degree of a morphism \( V \to \Ag \) is given below).
Note that this definition is not the direct generalisation of the definition of asymmetric curves in \( \cA_1^n \) from \cite{habegger-pila:beyond-ao} -- if we were to generalise our definition to a curve \( V \subset \Ag^n \), we would demand that the degrees of all the coordinate projections \( V \to \Ag \) should be distinct, while the definition in \cite{habegger-pila:beyond-ao} allows one degree to occur twice.
The reason for this difference is because we look at points on \( V \subset \Ag \times \Ag \) which satisfy a single isogeny relation, while \cite{habegger-pila:beyond-ao} considered points on \( V \subset \cA_1^n \) satisfying two isogeny relations.

We define the degree of a morphism from a curve to \( \Ag \) as follows.
Let \( \bar\Ag \) denote the Baily--Borel compactification of \( \Ag \).
There is an ample line bundle \( \cL_{BB} \) on~\( \bar\Ag \) given by automorphic forms of weight~\( 1 \).
A morphism \( f \colon V \to \Ag \) from an irreducible complex algebraic curve \( V \) to \( \Ag \) induces a morphism \( \bar{f} \colon \bar{V} \to \bar\Ag \) where \( \bar{V} \) is a smooth projective curve birational to \( V \).
We define the \defterm{degree} of \( f \colon V \to \Ag \) to be the degree of the line bundle \( \bar{f}^* \cL_{BB} \) on \( \bar{V} \).

Let \( \Sigma' \) denote the set of points \( (s_1, s_2) \in \Ag \times \Ag \) such that \( A_{s_1} \) is isogenous to \( A_{s_2} \) (with no restriction no \( \End A_{s_1} \)).
For each point \( s = (s_1, s_2) \in \Sigma' \), define the complexity \( N'(s) \) to be the minimum degree of an isogeny \( A_{s_1} \to A_{s_2} \).
If \( s \in \Sigma \), then the isogeny \( A_{s_1} \to A_{s_2} \) of minimum degree is polarised and consequently in this case \( N'(s) = N(s) \), where \( N(s) \) is the function defined in \cref{zp-translates-setup}.

For each point \( s \in \Ag(\Qalg) \), write \( h_F(s) \) for the absolute logarithmic semistable Faltings height of the associated abelian variety \( A_s \).

We use the asymmetry hypothesis via the following lemma, which adapts the proof of \cite[Lemma~4.2]{habegger-pila:beyond-ao}.

\begin{lemma} \label{faltings-height-bound}
Let \( V \subset \Ag \times \Ag \) be an asymmetric irreducible algebraic curve defined over \( \Qalg \) which is not contained in any proper special subvariety of \( \Ag \times \Ag \).

There exists a constant \( \newC{faltings-height-bound-multiplier} \) such that, for all but finitely many points \( s = (s_1, s_2) \in V \cap \Sigma' \),
\[ h_F(s_1)  \leq  \refC{faltings-height-bound-multiplier} \log N'(s). \]
\end{lemma}

\begin{proof}
Observe first that each point of \( V \cap \Sigma' \) is defined over \( \Qalg \).
This is because \( \Sigma' \) is contained in a union of proper special subvarieties of \( \Ag \times \Ag \).
Since \( V \) is a Hodge generic curve, its intersection with each proper special subvariety is finite.
Since both \( V \) and the special subvarieties are algebraic varieties defined over \( \Qalg \), we deduce that all the points of their intersections are defined over \( \Qalg \).

Let \( \bar{V} \) be a smooth projective curve birational to \( V \) and let \( \bar{p}_1 \), \( \bar{p}_2 \) denote the morphisms \( \bar{V} \to \bar\Ag \) induced by the projections \( V \subset \Ag \times \Ag \to \Ag \).
Let
\[ \cL_i = \bar{p}_i^* \cL_{BB} \]
for \( i \in \{ 1, 2 \} \), and let \( d_i \) be the degree of the line bundle \( \cL_i \).
Since \( \bar{V} \) is a curve, \( \cL_1^{\otimes d_2} \) is algebraically equivalent to \( \cL_2^{\otimes d_1} \).

Assume without loss of generality that \( \bar{p}_1 \) is non-constant.
Since \( \cL_{BB} \) is ample, it follows that \( \cL_1 \) is ample.
Hence \( d_1 > 0 \).

For the line bundles \( \cL_1 \) and \( \cL_2 \), choose height functions \( h_{\cL_i} \colon \bar{V}(\Qalg) \to \bR \) as in Weil's height machine \cite[Theorem~B.3.6]{hindry-silverman}.
By the additivity and algebraic equivalence properties of height functions, we have
\[ (d_2/d_1) h_{\cL_1}(s) - h_{\cL_2}(s)  =  \ro(h_{\cL_1}(s)) \]
as \( s \) runs over points in \( \bar{V}(\Qalg) \) with \( h_{\cL_1}(s) \to \infty \) (using the fact that \( \cL_1 \) is ample).

Also choose a height function \( h_{BB} \colon \bar\Ag(\Qalg) \to \bR \) associated with the Baily--Borel line bundle \( \cL_{BB} \) on \( \Ag \).
By the functoriality property in the height machine,
\[ h_{\cL_i}(s)  =  h_{BB}(p_i(s))  + \rO(1). \]
As proved in \cite[p.~356]{faltings:endlichkeitssaetze}, for each \( i = 1 \) and \( 2 \),
\[ \abs{h_F(p_i(s)) -  h_{BB}(p_i(s))}  = \rO(\log  h_{BB}(p_i(s))). \]
Combining the above relations between heights, we deduce that
\begin{equation} \label{eqn:faltings-height-degree-inequality}
(d_2/d_1) h_F(p_1(s)) - h_F(p_2(s)) = \ro(h_F(p_1(s))).
\end{equation}

Furthermore by \cite[Lemma~5]{faltings:endlichkeitssaetze},
there is a constant \( \newC{faltings-isogeny-multiplier} \) such that
\begin{equation} \label{eqn:faltings-height-isogeny-inequality}
\abs{h_F(p_1(s)) - h_F(p_2(s))}  \leq  \refC{faltings-isogeny-multiplier} \log N'(s).
\end{equation}

Because \( V \) is asymmetric, \( d_1 \neq d_2 \).
Hence we can combine \eqref{eqn:faltings-height-degree-inequality} and \eqref{eqn:faltings-height-isogeny-inequality} to obtain
\begin{equation*}
h_F(p_1(s))  =  \newC{faltings-height-multiplier} \log N'(s) + \ro(h_F(p_1(s))).
\end{equation*}
It follows that there is some constant \( \newC{faltings-height-o-bound} \) such that, for every \( s \in V \cap \Sigma \), either \( h_F(p_1(s)) \leq \refC{faltings-height-o-bound} \) or
\begin{equation} \label{eqn:faltings-height-final-bound}
h_F(p_1(s))  \leq  2 \refC{faltings-height-multiplier} \log N'(s).
\end{equation}

By \cite[Lemma~3]{faltings:endlichkeitssaetze}, there are only finitely many points \( s_1 \in \Ag \) such that \( h_F(s_1) \leq \refC{faltings-height-o-bound} \).
Since the projection \( p_{1|V} \) is quasi-finite, it follows that there are only finitely many points \( s \in V \) such that \( h_F(p_1(s)) \leq \refC{faltings-height-o-bound} \).
We conclude that \eqref{eqn:faltings-height-final-bound} holds for all but finitely many \( s \in V \cap \Sigma \), as required.
\end{proof}

\subsection{Large Galois orbits}

The proof of \cref{zp-isog-asymmetric} uses the following theorem of Masser and W\"ustholz.

\begin{theorem} \cite{masser-wustholz:isogeny-avs} \label{mw-bound}
Given positive integers \( g \) and~\( d \), there exist constants \( \newC{mw-multiplier} \) and \( \newC{mw-exponent} \) depending only on \( g \) (not on \( d \)) such that for all principally polarised abelian varieties \( A \) and \( B \) of dimension at most~\( g \), defined over a number field of degree at most \( d \),
if \( A \) and \( B \) are isogenous, then there is an isogeny \( A \to B \) of degree at most
\[ \refC{mw-multiplier} \, (d \cdot \max(1, h_F(A)))^{\refC{mw-exponent}}. \]
\end{theorem}

Note that the isogeny of bounded degree whose existence is asserted by \cref{mw-bound} is not necessarily a polarised isogeny, even if we know \textit{a priori} that there exists a polarised isogeny \( A \to B \).
Consequently the following bound is in terms of \( N'(s) \), not \( N(s) \).

\begin{proposition} \label{ag:galois-orbits-asymmetric}
Let \( V \subset \Ag \times \Ag \) be an asymmetric irreducible algebraic curve defined over a number field~\( L \).
Assume that \( V \) is not contained in any proper special subvariety of \( \Ag \times \Ag \).

There exist constants \( \newC{goh-ag-asymmetric-multiplier}, \newC{goh-ag-asymmetric-exponent} > 0 \) depending only on \( V \) and \( g \) such that, for every \( s \in V \cap \Sigma' \),
\[ \#(\Aut(\bC/L) \cdot s)  \geq  \refC{goh-ag-asymmetric-multiplier} \, N'(s)^{\refC{goh-ag-asymmetric-exponent}}. \]
\end{proposition}

\begin{proof}
Let \( s = (s_1, s_2) \in V \cap \Sigma' \).

Because \( \Ag \) is only a coarse moduli space for principally polarised abelian varieties, \( A_{s_1} \) and \( A_{s_2} \) need not have models over the field \( \bQ(s) \).
However, there is a fine moduli space \( \Ag(3) \) for principally polarised abelian varieties with full \( 3 \)-torsion level structure.
There is a finite surjective morphism \( \Ag(3) \to \Ag \) forgetting the level structure.
Let \( s' \) be a point in \( \Ag(3) \times \Ag(3) \) which maps to \( s \in \Ag \times \Ag \).
Then \( A_{s_1} \) and \( A_{s_2} \) have models defined over the field \( \bQ(s') \).

Furthermore \( [\bQ(s'):\bQ(s)] \) is bounded by a constant, namely the degree of the map \( \Ag(3) \times \Ag(3) \to \Ag \times \Ag \).
Hence there is a constant \( \newC{x} \) such that
\begin{equation} \label{ineqn:s-degree}
[\bQ(s'):\bQ]  \leq  [\bQ(s'):\bQ(s)] [L(s):L] [L:\bQ]  \leq \refC{x} \, [L(s):L].
\end{equation}

Since \( A_{s_1} \) and \( A_{s_2} \) are both defined over \( \bQ(s') \), \cref{mw-bound} tells us that
\[ N'(s)  \leq  \refC{mw-multiplier} \, ([\bQ(s'):\bQ] \cdot \max(1, h_F(A)))^{\refC{mw-exponent}}. \]
Using \cref{faltings-height-bound} (which depends on the asymmetry hypothesis), we deduce that
there is a constant \( \newC{new-mw-multiplier} \) such that
\[ N'(s)  \leq  \refC{new-mw-multiplier} \, ([\bQ(s'):\bQ] \cdot \max(1, \log N'(s)))^{\refC{mw-exponent}} \]
for all but finitely many \( s \in V \cap \Sigma \).
We can remove the finitely many exceptions and the \( \log N'(s) \) factor by adjusting the constants, giving
\[ N'(s)  \leq  \newC* \, [\bQ(s'):\bQ]^{\newC*}. \]
Combining this with inequality~\eqref{ineqn:s-degree} and noting that \( \# (\Aut(\bC/L) \cdot s) = [L(s):L] \) proves the proposition.
\end{proof}

\begin{proof}[Proof of \cref{zp-isog-asymmetric}]
We apply \cref{zp-translates-conditional} to \( S = \Ag \times \Ag \).
Let \( (\gG, X) \) be the Shimura datum \( (\gGSp_{2g}, \Hg^{\pm}) \times (\gGSp_{2g}, \Hg^{\pm}) \) and let \( (\gH, X_\gH) \subset (\gG, X) \) be the diagonal sub-Shimura datum.

Let
\[ \Omega = \{ (1, \gamma) \in \gG(\bQ)_+ : \gamma \in \gGSp_{2g}(\bQ)_+ \cap \rM_{2g}(\bZ) \}. \]
For each \( \gamma \in \gGSp_{2g}(\bQ)_+ \), the special subvariety \( S_{\gH,(1,\gamma)} \) defined in \cref{zp-translates-setup} is equal to the Hecke correspondence \( T_\gamma \).
Thus by \cref{ag:isogeny-to-hecke},
\[ \bigcup_{\gamma \in \Omega} S_{\gH,\gamma} = \{ (s_1, s_2) \in \Ag \times \Ag : \text{there exists a polarised isogeny } A_{s_1} \to A_{s_2} \}. \]

According to the definition in \cref{zp-isog-asymmetric},
\[ \Sigma = \{ (s_1, s_2) \in \bigcup_{\gamma \in \Omega} S_{\gH,\gamma} : \End A_{s_1} \cong \bZ \}. \]
If \( s = (s_1, s_2) \in \Sigma \), then all isogenies \( A_{s_1} \to A_{s_2} \) are polarised.
In particular, this applies to the isogeny of minimum degree and so by \cref{ag:isogeny-to-hecke}, the complexity function \( N(s) \) defined in \cref{zp-translates-setup} is the same as the function \( N'(s) \) used in \cref{ag:galois-orbits-asymmetric}.
(Because we only consider matrices \( \gamma \) with integer entries and positive determinant, the definition in \cref{zp-translates-setup} simplifies to \( N(\gamma) = \det \gamma \).)

\Cref{conj:galois-orbits} holds in this setting by \cref{ag:galois-orbits-asymmetric} and \cref{conj:field-of-definition} holds by \cref{ag:hecke-components-bound}.
Therefore we can apply \cref{zp-translates-conditional} to prove \cref{zp-isog-asymmetric}.
\end{proof}

\vskip 0.2em

\section{Unconditional Zilber--Pink: transcendental field of definition} \label{sec:transcendental}

In this section, we prove \cref{conj:galois-orbits} (large Galois orbits) for points in the intersection of suitable Hecke correspondences in \( \Ag \times \Ag \) with a curve \( V \) whose field of definition has positive transcendence degree over the field of definition of \( p_1(V) \subset \Ag \), where \( p_1 \colon \Ag \times \Ag \to \Ag \) denotes projection onto the first factor.
This implies \cref{intro:zp-isog-transcendental}.

Let \( \Sigma \) be the set defined in \cref{zp-isog-transcendental}.
The hypothesis on the fields of definition of \( V \) and \( p_1(V) \) allows us to relate the size of the Galois orbits of points \( s \in V \cap \Sigma \) to the gonality of suitable covers of~\( p_1(V) \).
Note that in \cite[Theorem~1.4]{pila:modular-fermat} (the analogous statement for \( g = 1 \)), the Zariski closure of \( p_1(V) \) is \( \cA_1 \) itself, which is defined over~\( \bQ \).
Hence in \cite[Theorem~1.4]{pila:modular-fermat}, the hypotheses on the field \( K \) of \cref{zp-isog-transcendental} were replaced by the simpler hypothesis that \( V \) is not defined over \( \Qalg \).


We bound the gonality of covers of \( p_1(V) \) using the theorem of Ellenberg, Hall and Kowalski \cite{ehk:expanders} together with a super-approximation result of Salehi Golsefidy \cite{sg:super-approximation} (which extends an earlier result of Salehi Golsefidy and Varj\`u \cite{sgv:expansion}).
In \cite{pila:modular-fermat}, these covers were modular curves so it was possible to use the simpler gonality bound of Abramovich \cite{abramovich:gonality}.
The restriction to isogenies whose degrees are \( b \)-th-power-free is necessary in order to apply \cite{sg:super-approximation}.

This gives a bound for Galois orbits of a points in the intersection of \( V \) with a Hecke correspondence \( T_\gamma \) in terms of the index of the congruence subgroup \( \Gamma \cap \gamma \Gamma \gamma^{-1} \) which is valid for all Shimura varieties (\cref{galois-orbits-transcendental}).
In order to convert this to a bound in terms of \( N(\gamma) \), as in \cref{conj:galois-orbits}, we have to use \cref{ag:index-complexity-bound} which applies only to \( \Ag \).
When deducing \cref{zp-isog-transcendental}, we also make use of other results from section~\ref{sec:hecke-correspondences} which apply only to Hecke correspondences in \( \Ag \times \Ag \).
If we could generalise all the results of section~\ref{sec:hecke-correspondences} to other Shimura varieties, then we could obtain a generalisation of \cref{zp-isog-transcendental}.


\subsection{Expansion of graphs}

An \defterm{expander family} is a family of connected graphs (in which self-loops and multiple edges are permitted) such that
\begin{enumerate}
\item every graph in the family is \( r \)-regular for some fixed integer~\( r \); and
\item there exists \( \epsilon > 0 \) such that for every graph \( G \) in the family,
\[ \min \{ \abs{\partial X}/\abs{X} : X \subset V(G), 0 < \abs{X} \leq \abs{V(G)}/2 \} \geq \epsilon. \]
\end{enumerate}
Here \( \partial X \) denotes the set of edges of \( G \) which have one endpoint in \( X \) and the other endpoint not in \( X \).
This is the definition used in \cite{ehk:expanders}.
The definition in \cite{sgv:expansion} is the same except that it does not mention the conditions ``connected'' and ``\( r \)-regular.''
The Cayley graphs considered in \cite{sgv:expansion} and \cite{sg:super-approximation} are automatically connected and \( r \)-regular so this difference does not matter.

We will use the following results on expander families.

\begin{theorem} \cite[Theorem~1]{sg:super-approximation} \label{sg-expansion}
Let \( \Gamma \subset \gGL_n(\bZ) \) be the subgroup generated by a finite symmetric set \( \Delta \).
For each positive integer \( q \), let \( \pi_q \) denote the natural map \( \gGL_n(\bZ) \to \gGL_n(\bZ/q\bZ) \).
Let \( b \geq 2 \) be a positive integer.

The Cayley graphs \( \Cay(\pi_q(\Gamma), \pi_q(\Delta)) \) form an expander family as \( q \) ranges over the \( b \)-th-power-free positive integers if and only if the identity component of the Zariski closure of \( \Gamma \) is perfect.
\end{theorem}

Given a field \( K \) and an algebraic curve \( C/K \), let \( \gon_K(C) \) denote the \( K \)-gonality of \( C \), that is, the minimum degree of a non-constant \( K \)-rational map \( C \dashrightarrow \bP^1 \).

\begin{theorem} \cite[Theorem~8(b)]{ehk:expanders} \label{ehk-expanders}
Let \( U \) be a smooth connected algebraic curve over \( \bC \).
Let \( \{ U_i : i \in I \} \) be an infinite family of connected \'etale covers of \( U \).
Pick a point \( u \in U \), and for each \( i \), let \( u_i \) be a point of \( U_i \) which maps to \( u \in U \).

Let \( \Delta \) be a finite symmetric generating set of \( \pi_1(U, u) \).

If the family of Cayley--Schreier graphs
\[ \Cay(\pi_1(U, u) / \pi_1(U_i, u_i), \Delta) \] is an expander family,
then there exists a constant \( \newC{ehk-multiplier} > 0 \) such that
\[ \gon_\bC(U_i) \geq \refC{ehk-multiplier} \, [ \pi_1(U, u) : \pi_1(U_i, u_i) ] \]
for every \( i \in I \).
\end{theorem}

\subsection{Gonality growth for congruence subgroup covers}

Let \( (\gG, X) \) be a Shimura datum and let \( X^+ \) be a connected component of~\( X \).
For each congruence subgroup \( \Gamma \subset \gG(\bQ)_+ \), write \( S_\Gamma \) for the Shimura variety component \( \Gamma \bs X^+ \).


Fix a congruence subgroup \( \Gamma \subset \gG(\bQ)_+ \) and write \( S = S_\Gamma \).
Let \( W \subset S \) be an irreducible algebraic subvariety which is not contained in any proper special subvariety of \( S \).
When \( W \) is a curve, we prove a lower bound for the gonality of irreducible components of \( W \times_S S_{\Gamma'} \) as \( \Gamma' \) runs over the subgroups of \( \Gamma \) which contain a principal congruence subgroup of \( b \)-th-power-free level.
This gonality bound (\cref{gonality-degree-bound}) is restricted to curves, but \( W \) may be of any dimension in the intermediate results \cref{cayley-graphs-expand,strong-approximation}.
In order to prove \cref{zp-isog-transcendental}, we will apply the gonality bound to the Zariski closure of \( p_1(V) \).

In order to define principal congruence subgroups of \( \Gamma \), we choose a faithful representation \( \rho \colon \gG \to \gGL_{n,\bQ} \) such that \( \rho(\Gamma) \subset \gGL_n(\bZ) \).
(The principal congruence subgroups, and hence the meaning of the condition ``\( \Gamma \cap \gamma \Gamma \gamma^{-1} \) contains a principal congruence subgroup of \( b \)-th-power-free level,'' depend on the choice of \( \rho \).)
For each positive integer~\( q \), let \( \Gamma(q) \) denote the kernel of the map \( \Gamma \to \gGL_n(\bZ/q\bZ) \) induced by \( \rho \).

For each congruence subgroup \( \Gamma' \subset \Gamma \), let
\[ W_{\Gamma'} = W \times_S S_{\Gamma'}. \]
For each \( q \in \bN \), let \( W_{q,1}, \dotsc, W_{q,r_q} \) denote the irreducible components of \( W_{\Gamma(q)} \).

Let \( W^\sm \) denote the smooth locus of \( W \), and similarly for \( W_{q,i}^\sm \).
Pick a base point \( w \in W^\sm \).
For each \( q \) and \( i \), pick \( w_{q,i} \in W_{q,i}^\sm \) which maps to \( w \in W \).

\begin{lemma} \label{monodromy-properties}
Assume that the congruence subgroup \( \Gamma \) is neat.

There exists a subgroup \( \Gamma_W \subset \Gamma \) such that
\begin{enumerate}[(i)]
\item the identity component of the Zariski closure of \( \Gamma_W \) is a normal subgroup of \( \gG^\der \); and
\item for each positive integer \( q \) and each irreducible component \( W_{q,i} \subset W_{\Gamma(q)} \), \( \pi_1(W^\sm, w) / \pi_1(W_{q,i}^\sm, w_{q,i}) \) is in bijection with \( \Gamma_W / \Gamma_W \cap \Gamma(q) \).
\end{enumerate}
\end{lemma}

\begin{proof}
The inclusion \( W^\sm \to W \to S \) induces a homomorphism of fundamental groups \( \iota \colon \pi_1(W^\sm, w) \to \pi_1(S, w) \).
Since \( \Gamma \) is neat, \( X^+ \) is the universal cover of \( S \).
Therefore choosing a point \( x \in X^+ \) which maps to \( w \in S \) induces an isomorphism \( f_x \colon \pi_1(S, w) \to \Gamma \).

Let \( \Gamma_W \) denote the image of \( f_x \circ \iota \colon \pi_1(W^\sm, w) \to \Gamma \).

The representation \( \rho \colon \gG \to \gGL_{n,\bQ} \) induces a variation of \( \bZ \)-Hodge structures on \( X^+ \) (with underlying lattice \( \bZ^n \)).
Since \( \Gamma \) is neat and \( \rho(\Gamma) \subset \gGL_n(\bZ) \), this descends to a variation of Hodge structures \( V_\bZ \) on \( S \).
Since \( W \) is not contained in any proper special subvariety of \( S \), the generic Mumford--Tate group of \( V_{\bZ|W^\sm} \) is \( \rho(\gG) \).
Therefore by \cite[Theorem~1]{andre:fixed-part}, the identity component of the Zariski closure of the monodromy group \( \Gamma_W \) is a normal subgroup of \( \gG^\der \).

The action of \( \pi_1(W^\sm, w) \) on \( \{ w \} \times_S S_{\Gamma(q)} \) factors through \( \iota \).
Hence for each~\( i = 1, \dotsc, r_q \), we have
\[ \pi_1(W_{q,i}^\sm, w_{q,i})  =  \Stab_{\pi_1(W^\sm, w)} (w_{q,i})  =  \iota^{-1}(\Stab_{\pi_1(S, w)} (w_{q,i})). \]
For each point \( w' \in \{ w \} \times_S S_{\Gamma(q)} \), the stabiliser \( \Stab_{\pi_1(S, w)} (w') \) is conjugate to \( f_x^{-1}(\Gamma(q)) \) in \( \pi_1(S, w) \).
Because \( \Gamma(q) \) is a normal subgroup of \( \Gamma \), we deduce that in fact \( \Stab_{\pi_1(S, w)} (w') = f_x^{-1}(\Gamma(q)) \) for every \( w' \in \{ w \} \times_S S_{\Gamma(q)} \).
Therefore
\[ \pi_1(W_{q,i}^\sm, w_{q,i})  =  (f_x \circ \iota)^{-1}(\Gamma(q)) \]
and so \( f_x \circ \iota \) induces a bijection
\[ \pi_1(W^\sm, w) / \pi_1(W_{q,i}^\sm, w_{q,i})  \to  \Gamma_W / \Gamma_W \cap \Gamma(q).
\qedhere
\]
\end{proof}

The fundamental group \( \pi_1(W^\sm, w) \) is finitely generated because \( W^\sm \) is a smooth quasi-projective complex algebraic variety.
Choose a finite symmetric generating set \( \Delta \) for \( \pi_1(W^\sm, w) \)
(\defterm{symmetric} means that \( g \in \Delta \Rightarrow g^{-1} \in \Delta \)).

\begin{corollary} \label{cayley-graphs-expand}
Assume that \( \Gamma \) is neat.

The Cayley graphs
\[ \Cay(\pi_1(W^\sm, w) / \pi_1(W_{q,i}^\sm, w_{q,i}), \Delta) \]
form an expander family as \( q \) runs over the \( b \)-th-power-free positive integers, and \( W_{q,i} \) runs over the connected components of \( W_{\Gamma(q)} \) for each \( q \).
\end{corollary}

\begin{proof}
Let \( \Gamma_W \subset \Gamma \) be the monodromy group of \( W^\sm \) as in \cref{monodromy-properties}.
The identity component of the Zariski closure of \( \Gamma_W \) is semisimple and therefore perfect.
Hence we can apply \cite[Theorem~1]{sg:super-approximation} (\cref{sg-expansion}) to \( \Gamma_W \) to obtain the desired conclusion.
\end{proof}

We have shown that the Cayley graphs associated with the irreducible components of \( W_{\Gamma(q)} \) form an expander family.
In order to use this, we also need the following consequence of Nori's strong approximation theorem which tells us that the number of irreducible components of \( W_{\Gamma(q)} \) is bounded.
This lemma, in combination with \cref{monodromy-properties}(i), is similar to a step in the proof of \cite[Theorem~5.1]{edixhoven-yafaev}.

\begin{lemma} \label{strong-approximation}
Let \( \gH \subset \gGL_n \) be a semisimple \( \bQ \)-algebraic subgroup.

Let \( \Gamma = \gH(\bQ) \cap \gGL_n(\bZ) \).
For each positive integer \( q \), let \( \Gamma(q) \) be the kernel of the natural map \( \Gamma \to \gGL_n(\bZ/q\bZ) \).

Let \( \Gamma_W \subset \Gamma \) be a finitely generated subgroup which is Zariski dense in \( \gH \).

Then \( [ \Gamma / \Gamma(q) : \Gamma_W / \Gamma_W \cap \Gamma(q) ] \) is bounded as \( q \) runs over all positive integers.
\end{lemma}

\begin{proof}
Let \( \bar\Gamma \) denote the closure of \( \Gamma \) in \( \gGL_n(\hat\bZ) \) (with respect to the profinite topology), and similarly define \( \bar\Gamma_W \).
For each positive integer \( m \), let \( K(m) \) be the kernel of the natural map \( \gGL_n(\hat\bZ) \to \gGL_n(\bZ/m\bZ) \).

By \cite[Theorem~5.3]{nori:strong-approx} \( \bar\Gamma_W \) is open in \( \bar\Gamma \).
Note that \cite{nori:strong-approx} does not contain a proof of this theorem; an outline of a proof can be found at \cite[Remark~5.3]{edixhoven-yafaev}.
Hence there exists \( m \) such that \( \bar\Gamma \cap K(m) \subset \bar\Gamma_W \), and so
\[ \Gamma(m) \subset \bar\Gamma \cap K(m) \subset \bar\Gamma_W. \]

Therefore, for every \( q \in \bN \) and every \( \gamma \in \Gamma(m) \), the open set \( \gamma.K(q) \subset \gGL_n(\hat\bZ) \) intersects \( \Gamma_W \).
It follows that \( \gamma.\Gamma(q) \) is in the image of \( \Gamma_W \to \Gamma/\Gamma(q) \).

Therefore we have
\[ [ \Gamma / \Gamma(q) : \Gamma_W / \Gamma_W \cap \Gamma(q) ]  \leq  [ \Gamma : \Gamma(m) ] \]
for every \( q \), where \( m \) is independent of \( q \).
\end{proof}

Now we restrict to \( W \) being a curve, in order to apply \cite{ehk:expanders}.

\begin{proposition} \label{gonality-degree-bound}
Let \( W \subset S \) be an irreducible algebraic curve.

There exists a constant \( \newC{gonality-degree-multiplier} > 0 \) such that, for every congruence subgroup \( \Gamma' \subset \Gamma \) which contains \( \Gamma(q) \) for some \( b \)-th-power-free positive integer \( q \), each irreducible component \( Z \) of \( W_{\Gamma'} \) satisfies
\[ \gon_\bC(Z)  \geq  \refC{gonality-degree-multiplier} \, [\Gamma : \Gamma']. \]
\end{proposition}

\begin{proof}
Note first that we may replace \( \Gamma \) by a finite index subgroup \( \Gamma_0 \).
At the same time we replace \( W \) by an irreducible component of \( W_{\Gamma_0} \), \( \Gamma' \) by \( \Gamma' \cap \Gamma_0 \) (which contains \( \Gamma(q) \cap \Gamma_0 \)) and \( Z \) by an irreducible component of \( W_{\Gamma' \cap \Gamma_0} \).
These replacements change \( \gon_\bC(Z) \) and \( [\Gamma: \Gamma'] \) by bounded factors, so they do not affect the conclusion of the proposition.
Thus we may assume that \( \Gamma \) is neat.
This ensures that \( S_{\Gamma'} \to S \) is \'etale.

Consider the family of connected \'etale covers \( W_{q,i}^\sm \to W^\sm \), as \( q \) runs over the \( b \)-th-power-free positive integers and \( W_{q,i} \) runs over the connected components of \( W_{\Gamma(q)} \) for each \( q \).
Thanks to the expansion result of \cref{cayley-graphs-expand}, we can apply \cite[Theorem~8(b)]{ehk:expanders} (\cref{ehk-expanders}) to get
\[ \gon_\bC(W_{q,i}^\sm)
   \geq  \refC{ehk-multiplier} \, [ \pi_1(W^\sm, w) : \pi_1(W_{q,i}^\sm, w_{q,i}) ].
\]

Let \( \Gamma_W \subset \Gamma \) be the monodromy group of \( W^\sm \) as in \cref{monodromy-properties}.
By \cref{monodromy-properties,strong-approximation} we have
\[ [ \pi_1(W^\sm, w) : \pi_1(W_{q,i}^\sm, w_{q,i}) ]
   =  [\Gamma_W : \Gamma_W \cap \Gamma(q)]
   \geq  \frac{[\Gamma : \Gamma(q)]} {\refC{strong-approx-bound}}
\]
for some constant \( \newC{strong-approx-bound} > 0 \).

Given an irreducible component \( Z \) of \( W_{\Gamma'} \), choose a \( b \)-th-power-free positive integer~\( q \) such that \( \Gamma(q) \subset \Gamma' \).
Then some irreducible component \( W_{q,i} \) of \( W_{\Gamma(q)} \) maps onto~\( Z \).
We have
\[ \deg(W_{q,i} \to Z)  \leq  \deg(S_{\Gamma(q)} \to S_{\Gamma'})  =  [\Gamma' : \Gamma(q)]. \]

Noting that \( \gon_\bC(W_{q,i}^\sm) = \gon_\bC(W_{q,i}) \) because \( W_{q,i}^\sm \) is a dense open subset of~\( W_{q,i} \), we can combine the above inequalities to get
\[ \gon_\bC(Z)
   \geq  \frac{\gon_\bC(W_{q,i})} {\deg(W_{q,i} \to Z)}
   \geq  \frac{\refC{ehk-multiplier} \, [\Gamma : \Gamma(q)]} {\refC{strong-approx-bound} \, [\Gamma' : \Gamma(q)]}
   \geq  \frac{\refC{ehk-multiplier}} {\refC{strong-approx-bound}} \, [\Gamma : \Gamma'].
\qedhere
\]
\end{proof}

\subsection{Large Galois orbits}

Before proving our large Galois orbits results, we need a lemma relating Galois orbits of transcendental points of an algebraic curve to the gonality of the curve.
In this lemma and its proof, all fields are considered as subfields of \( \bC \) and algebraic closures are taken inside \( \bC \).

\begin{lemma} \label{galois-gonality-bound}
Let \( K \) be an algebraically closed subfield of \( \bC \).
Let \( L \) be a finitely generated extension of \( K \) inside \( \bC \) such that \( \trdeg(L/K) \geq 1 \).

There exists a constant \( \newC{galois-gonality-multiplier} > 0 \) depending only on \( K \) and \( L \) such that,
for every irreducible algebraic curve \( Z \) defined over \( K \) and every point \( s \in Z(\Lalg) \setminus Z(K) \),
\[ \# (\Aut(\bC/L) \cdot s)  \geq  \refC{galois-gonality-multiplier} \, \gon_\bC(Z). \]
\end{lemma}

\begin{proof}
Choose a transcendence basis \( t_1, \dotsc, t_r \) for \( L/K \).
We make this choice once, independent of \( Z \) and \( s \).
Then \( [L:K(t_1, \dotsc, t_r)] \) is a finite constant.

Because \( s \in Z(\Lalg) \setminus Z(K) \), there exists \( m \in \{ 1, \dotsc, r \} \) such that \( s \) can be defined over an algebraic extension of \( K(t_1, \dotsc, t_m) \) but not over an algebraic extension of \( K(t_1, \dotsc, t_{m-1}) \).
Let \( M = K(t_1, \dotsc, t_{m-1}) \).
Let \( M(t_m, s) \) denote the smallest extension of \( M(t_m) \) inside \( \Mtmalg \) over which \( s \) is defined.
By Galois theory,
\[ [M(t_m, s):M(t_m)] = \# (\Gal(\Mtmalg/M(t_m)) \cdot s). \]

Consider the restriction maps
\[ \Aut(\bC/K(t_1, \dotsc, t_r)) \to \Gal(\overline{K(t_1, \dotsc, t_r)}/K(t_1, \dotsc, t_r)) \to \Gal(\Mtmalg/M(t_m)). \]
The first is surjective because \( \bC \) is algebraically closed and the second is surjective because \( K(t_1, \dotsc, t_r) = M(t_m, \dotsc, t_r) \) is a regular extension of \( M(t_m) \).
Therefore
\begin{align*}
    \# (\Gal(\Mtmalg / M(t_m)) \cdot s)
  & =  \# (\Aut(\bC/K(t_1, \dotsc, t_r)) \cdot s)
\\& \leq  [L:K(t_1, \dotsc, t_r)] \, \# (\Aut(\bC/L) \cdot s).
\end{align*}

Let \( M(Z) \) denote the field of rational functions on \( Z \) with values in \( M \).
Because \( Z \) is a curve, the locus of indeterminacy of any rational function in \( M(Z) \) is an \( M \)-scheme of dimension zero.
Since \( s \) is not defined over an algebraic extension of \( M \), we deduce that \( s \) is not contained in the locus of indeterminacy of any element of \( M(Z) \).
Therefore ``evaluation at \( s \)'' gives a well-defined field homomorphism \( M(Z) \to M(t_m, s) \).

Since \( M(t_m, s) \) has transcendence degree~\( 1 \) over \( M \), it is isomorphic to the field of rational functions of some irreducible curve \( Y \) over \( M \).
The field homomorphism \( M(Z) \to M(t_m, s) \) corresponds to a dominant \( M \)-rational map \( Y \to Z \).
By \cite[Proposition~1.1(vii)]{poonen:gonality}, the existence of such a map implies that
\[ \gon_M(Z)  \leq  \gon_M(Y). \]

The inclusion \( M(t_m) \to M(t_m, s) \) corresponds to a dominant \( M \)-rational map \( Y \to \bP^1_M \), so by the definition of gonality,
\[ \gon_M(Y)  \leq  [M(t_m, s) : M(t_m)]. \]

Combining the above inequalities and equations, we get
\[ \gon_M(Z)  \leq  [M(t_m, s) : M(t_m)]  \leq  [L:K(t_1, \dotsc, t_r)] \, \# (\Aut(\bC/L) \cdot s). \]
Since \( \gon_\bC(Z)  \leq \gon_M(Z) \) and \( [L:K(t_1, \dotsc, t_r)] \) is constant, this proves the lemma.
\end{proof}

\begin{proposition} \label{galois-orbits-transcendental}
Let \( (\gG, X) \) be a Shimura datum and let \( S = \Gamma \bs X^+ \) be an associated Shimura variety component.
Let \( b \) be a positive integer.

Let \( V \subset S \times S \) be an irreducible algebraic curve such that there exists an algebraically closed field \( K \subset \bC \) satisfying the hypotheses (i) and (ii) of \cref{zp-isog-transcendental}.
Suppose that \( V \) is not contained in any proper special subvariety of \( S \times S \), and that \( p_{1|V} \) is not constant.

Let \(L \) be a finitely generated field over which \( V \) is defined.

There exists a constant \( \newC{goh-transcendental-multiplier} > 0 \) depending only on \( S \), \( V \) and \( L \) such that,
for every \( \gamma \in \gG(\bQ)_+ \) such that \( \Gamma \cap \gamma^{-1} \Gamma \gamma \) contains \( \Gamma(q) \) for some \( b \)-th-power-free integer \( q \) and for every \( s \in V \cap T_\gamma \),
\[ \#(\Aut(\bC/L) \cdot s)  \geq  \refC{goh-transcendental-multiplier} \, [\Gamma : \Gamma \cap \gamma^{-1} \Gamma \gamma]. \]
\end{proposition}

\begin{proof}
By hypothesis~(ii) from \cref{zp-isog-transcendental}, \( V \) is an algebraic curve not defined over \( K \).
Hence \( V \) contains only finitely many points defined over \( K \).
We can ignore finitely many points while proving the proposition, so we assume that \( s \) is not defined over \( K \).

Let \( W \) denote the Zariski closure of \( p_1(V) \) in \( S \).
By hypothesis~(i) from \cref{zp-isog-transcendental}, \( W \) is defined over \( K \).
Since \( p_{1|V} \) is not constant, \( \dim W = 1 \).

We interpret the Hecke correspondence \( T_\gamma \) in two ways.
Firstly, it is by definition a subset of \( S \times S \).
Secondly, it is isomorphic to the Shimura variety component \( (\Gamma \cap \gamma^{-1} \Gamma \gamma) \bs X^+ \), and the natural map \( (\Gamma \cap \gamma^{-1} \Gamma \gamma) \bs X^+ \to S \) is the same as the restriction to \( T_\gamma \) of \( p_1 \colon S \times S \to S \).
Therefore
\[ W \times_S T_\gamma = p_1^{-1}(W) \cap T_\gamma. \]
Consequently \( V \cap T_\gamma \subset W \times_S T_\gamma \) (where the fibre product notation \( \times_S \, T_\gamma \) refers to the projection \( p_{1|T_\gamma} \colon T_\gamma \to S \)).

Choose an irreducible component \( Z \subset W \times_S T_\gamma \) such that \( s \in Z \).
Since \( W \) is defined over \( K \) and \( T_\gamma \) is defined over \( \Qalg \subset K \), \( Z \) is defined over \( K \).
Because \( p_{1|T_\gamma} \colon T_\gamma \to S \) is a finite open map in the complex topology, \( \dim Z = 1 \).

Since \( V \) is Hodge generic in \( S \times S \), it is not contained in \( T_\gamma \).
Since \( V \) is an irreducible algebraic curve, it follows that \( V \cap T_\gamma \) is finite.
Since \( V \) is defined over \( L \) and \( T_\gamma \) is defined over \( \Qalg \), the intersection \( V \cap T_\gamma \) is defined over \( \Lalg \).
We conclude that \( s \in Z(\Lalg) \).

Therefore we can apply \cref{galois-gonality-bound} to \( s \in Z(\Lalg) \setminus Z(K) \).
We get
\[ \#(\Aut(\bC/L) \cdot s)  \geq  \refC{galois-gonality-multiplier} \, \gon_\bC(Z). \]
By \cref{gonality-degree-bound} we have
\[ \gon_\bC(Z)  \geq  \refC{gonality-degree-multiplier} \, [\Gamma \colon \Gamma \cap \gamma^{-1} \Gamma \gamma]. \]
Combining these inequalities proves the proposition.
\end{proof}

Using the notation of \cref{zp-isog-transcendental}, for each point \( (s_1, s_2) \in \Sigma \) let \( N(s) \) denote the smallest \( b \)-th-power-free positive integer \( n \) such that there exists a polarised isogeny \( A_{s_1} \to A_{s_2} \) of degree~\( n \).

\begin{corollary} \label{ag:galois-orbits-transcendental}
Let \( \Sigma \subset \Ag \times \Ag \) be the set defined in \cref{zp-isog-transcendental}.

Let \( V \subset \Ag \times \Ag \) be a curve satisfying the hypotheses of \cref{zp-isog-transcendental}.
Suppose that \( p_{1|V} \) is not constant.

Let \( L \) be a finitely generated field over which \( V \) is defined.

There exists a constant \( \newC{goh-ag-transcendental-multiplier} > 0 \) depending only on \( V \) and \( L \) such that,
for every \( s \in V \cap \Sigma \),
\[ \#(\Aut(\bC/L) \cdot s)  \geq  \refC{goh-ag-transcendental-multiplier} N(s)^{1/g}. \]
\end{corollary}

\begin{proof}
By \cref{ag:isogeny-to-hecke}, if \( s \in \Sigma \) then \( s \in T_\gamma \) for some \( \gamma \in \gGSp_{2g}(\bQ)_+ \cap \rM_{2g}(\bZ) \) with \( \det \gamma = N(s) \).
Using \cref{symplectic-elem-div}, we can see that \( \Gamma(q) \subset \Gamma \cap \gamma^{-1} \Gamma \gamma \) where \( q = \nu(\gamma) = N(s)^{1/g} \).
Because \( N(s) \) is \( b \)-th-power-free, \( q \) is \( \lceil b/g \rceil \)-th-power-free and so we can apply \cref{galois-orbits-transcendental} to all \( s \in V \cap T_\gamma \).
Combining this with \cref{ag:index-complexity-bound} proves the corollary.
\end{proof}

\begin{proof}[Proof of \cref{zp-isog-transcendental}]
The case in which \( p_{1|V} \) is constant is easily dealt with: let \( w \) be the image of \( p_{1|V} \).
By hypothesis~(i), \( w \in \Ag(K) \).
Then every abelian variety isogenous to \( A_w \) is defined over \( K \), so every point of \( V \cap \Sigma \) is defined over \( K \).
But by hypothesis~(ii), \( V \) is a curve not defined over \( K \), so \( V \cap \Sigma \) must be finite.

Otherwise, when \( p_{1|V} \) is not constant, we use \cref{zp-translates-conditional,ag:galois-orbits-transcendental}.
Let \( (\gG, X) \) be the Shimura datum \( (\gGSp_{2g}, \Hg^{\pm}) \times (\gGSp_{2g}, \Hg^{\pm}) \) and let \( (\gH, X_\gH) \subset (\gG, X) \) be the diagonal sub-Shimura datum.
Let
\[ \Omega = \{ (1, \gamma) \in \gG(\bQ)_+ : \gamma \in \gGSp_{2g}(\bQ)_+ \cap \rM_{2g}(\bZ) \text{ and } \det \gamma \text{ is \( b \)-th-power-free} \}. \]
For each \( \gamma \in \gGSp_{2g}(\bQ)^+ \), the special subvariety \( S_{\gH,(1,\gamma)} \) defined in \cref{zp-translates-setup} is equal to the Hecke correspondence \( T_\gamma \).
Thus by \cref{ag:isogeny-to-hecke}, \( \bigcup_{\gamma \in \Omega} S_{\gH,\gamma} \) is the same as the set \( \Sigma \) defined in \cref{zp-isog-transcendental}.
Furthermore the complexity functions \( N(s) \) defined in \cref{zp-translates-setup,ag:galois-orbits-transcendental} are the same.

\Cref{conj:galois-orbits} holds in this setting by \cref{ag:galois-orbits-transcendental} (which we can apply because \( p_{1|V} \) is not constant) and \cref{conj:field-of-definition} holds by \cref{ag:hecke-components-bound}.
Therefore we can apply \cref{zp-translates-conditional} to prove \cref{zp-isog-transcendental}.
\end{proof}

\bibliographystyle{amsalpha}
\bibliography{reduction3}

\end{document}